\documentclass[11pt,letterpaper]{article}
\usepackage[english]{babel}
\usepackage{amsmath}
\usepackage{amsfonts, mathrsfs}
\usepackage{amssymb}
\usepackage{verbatim}
\usepackage{graphicx}
\usepackage{color}
\usepackage{url}

\numberwithin{equation}{section}
\newtheorem{theo}{Theorem}[section]
\newtheorem{cor}[theo]{Corollary}
\newtheorem{prop}[theo]{Proposition}
\newtheorem{lemma}[theo]{Lemma}

\newtheorem{assum}[theo]{Assumption}
\newtheorem{defn}[theo]{Definition}

\newtheorem{remark}[theo]{Remark}
\newenvironment{proof}[1][Proof]{\textbf{#1.} }{\ \rule{0.5em}{0.5em}}

\newcommand{\var}{{\rm Var} \mspace{1mu}}

\renewcommand{\labelenumi}{\alph{enumi}.)}
\def\Indicator{\mathop{\hskip0pt{1}}\nolimits}
\newcommand{\sgn}{ \mspace{3mu} {\rm sgn}  \mspace{1mu}}

\usepackage{tikz}
\usepackage{lmodern}
\usetikzlibrary{arrows, decorations.markings}
\usepackage{pgfplots}

\begin{document}

\title{Ranking algorithms on directed configuration networks}

\author{Ningyuan Chen \\  {\small Columbia University}
\and
Nelly Litvak \\  {\small University of Twente}
\and
Mariana Olvera-Cravioto  \\ {\small Columbia University}
}

\maketitle

\begin{abstract}
This paper studies the distribution of a family of rankings, which includes Google's PageRank, on a directed configuration model. In particular, it is shown that the distribution of the rank of a randomly chosen node in the graph converges in distribution to a finite random variable $\mathcal{R}^*$ that can be written as a linear combination of i.i.d. copies of the endogenous solution to a stochastic fixed point equation of the form
$$\mathcal{R} \stackrel{\mathcal{D}}{=} \sum_{i=1}^{\mathcal{N}} \mathcal{C}_i \mathcal{R}_i + \mathcal{Q},$$
where $(\mathcal{Q}, \mathcal{N}, \{ \mathcal{C}_i\})$ is a real-valued vector with $\mathcal{N} \in \{0,1,2,\dots\}$, $P(|\mathcal{Q}| > 0) > 0$, and the $\{\mathcal{R}_i\}$ are i.i.d. copies of $\mathcal{R}$, independent of $(\mathcal{Q}, \mathcal{N}, \{ \mathcal{C}_i\})$.  Moreover, we provide precise asymptotics for the limit $\mathcal{R}^*$, which when the in-degree distribution in the directed configuration model has a power law imply a power law distribution for $\mathcal{R}^*$ with the same exponent.

\vspace{5mm}

\noindent {\em Kewywords:} PageRank, ranking algorithms, directed configuration model, complex networks, stochastic fixed-point equations, weighted branching processes, power laws.


\end{abstract}

\section{Introduction}

Ranking of nodes according to their centrality, or importance, in a complex network such as the Internet, the World Wide Web, and other social and biological networks, has been a hot research topic for several years in physics, mathematics, and computer science. For a comprehensive overview of the vast literature on rankings in networks we refer the reader to \cite{Langville2011google}, and more recently to~\cite{Boldi2014axioms} for a thorough up-to-date mathematical classification of centrality measures.

In this paper we analyze a family of ranking algorithms which includes Google's PageRank, the algorithm proposed by Brin and Page~\cite{Bri_Pag_98}, and which is arguably the most influential technique for computing rankings of nodes in large directed networks. The original definition of PageRank is the following. Let ${\cal G}_n=(V_n,E_n)$ be a directed graph, with a set of (numbered) vertices $V_n=\{1,\ldots,n\}$, and a set of directed edges $E_n$. Choose a constant $c\in (0,1)$, which is called a {\it damping factor}, and let ${\bf q}=(q_1,q_2,\ldots,q_n)$ be a {\it personalization} probability vector, i.e., $q_i \geq 0$ and $\sum_{i=1}^nq_i=1$. Denote by $d_i=|\{j: (i,j)\in E_n\}|$ the out-degree of node $i\in V_n$.  Then the PageRank vector ${\bf r}=(r_1,\ldots, r_n)$ is the unique solution to the following system of linear equations:
\begin{equation}
\label{eq:PR_def0}
r_i=\sum_{j: (j,i)\in E_n}\frac{c}{d_j}\,r_j +(1-c)q_i,\quad i=1,\ldots,n.
\end{equation}
Google's PageRank was designed to rank Web pages based on the network's structure, rather than their content. The idea behind (\ref{eq:PR_def0}) is that a page is important if many important pages have a hyperlink to it. Furthermore, by tuning the personalization values, $q_i$'s, one can, for instance, give preference to specific topics \cite{Haveliwala2002personalization} or penalize spam pages \cite{Gyongyi04}.

In the original definition, ${\bf r}$ is normalized so that $||{\bf r}||_1=1$, where the norm $||{\bf x}||_1=\sum_{i=1}^n|x_i|$ denotes the $l_1$ norm in $\mathbb{R}^n$. Since the average PageRank in ${\bf r}$ scales as $O(1/n)$, it is more convenient for our purposes to work with a scaled version of PageRank:
\[n{\bf r}=:{\bf R}=(R_1,R_2,\ldots,R_n).\]
Then, also using the notation $C_j$ for $c/d_j$, and notation $Q_i$ for $n(1-c)q_i$, we rewrite  (\ref{eq:PR_def0}) to obtain
\begin{equation}
\label{eq:PR_def2}
R_i=\sum_{j: (j,i)\in E_n}C_j\,R_j +Q_i,\quad i=1,\ldots,n.
\end{equation}
Throughout the paper, we will refer to ${\bf R}$ as  the PageRank vector and to ${\bf Q}=(Q_1,Q_2,\ldots,Q_n)$ as the personalization vector.

The basic definition (\ref{eq:PR_def0}) has many modifications and generalizations. The analysis in this paper will cover a wide range of them by allowing a general form of the coefficients in (\ref{eq:PR_def2}). For example, our model admits a random damping factor as studied in \cite{Constantine2009random_alpha}. Numerous applications of PageRank and its modifications include graph clustering~\cite{Andersen06}, spam detection~\cite{Gyongyi04}, and citation analysis~\cite{Chen06citations,Waltman2010eigenfactor}.

In real-world networks, it is often found that the fraction of nodes with (in- or out-) degree $k$ is $\approx c_0k^{-\alpha-1}$, usually $\alpha\in(1,3)$, see e.g., \cite{Bri_Pag_98, Newman2010networks}. Thus, a lot of research has been devoted to the study of random graph models with highly skewed, or scale-free, degree distributions. By now, classical examples are the Chung-Lu model~\cite{Chung2002CLmodel}, the Preferential Attachment model~\cite{Bollobas2001PA}, and the Configuration Model~\cite[Chapter 7]{HofstadRG}. New  models continue to appear, tuned to the properties of specific networks. For example, an interesting ``super-star'' model was recently developed to describe retweet graphs~\cite{Bhamidi2012superstar}. We refer to~\cite{HofstadRG,Durrett2007random,Newman2010networks} for a more detailed discussion of random graph models for complex networks.
In this paper we focus on the Directed Configuration Model as studied in~\cite{Chen_Olv_13}. Originally, an (undirected) Configuration Model is defined as a graph, randomly sampled from the set of graphs with a given degree sequence~\cite{Bollobas1980CM}. We emphasize that, to the best of our knowledge,~\cite{Chen_Olv_13} is the only paper that formally addresses the directed version of the Configuration Model and obtains its exact mathematical properties. We will provide more details in Section~\ref{S.DCM}.

From the work of Pandurangan et al.~\cite{Pandurangan2002PR}, and many papers that followed, the following hypothesis has always been confirmed by the data.

{\bf The power law hypothesis:} {\em If the in-degree distribution in a network follows a power law then the PageRank scores in this network will also follow a power law with the same exponent.}

The power law hypothesis is plausible because in (\ref{eq:PR_def0}) the number of terms in the summation on the right-hand side is just the in-degree of $i$, so the in-degree provides a `mean-field' approximation for PageRank~\cite{Fortunato2008approximating}. However, this argument is not exact nor accurate enough, which is confirmed by the fact that the top-ranked nodes in PageRank are not exactly those with the largest in-degrees \cite{Chen06citations,Volkovich2009complex,Vigna2014Kendall}. Exact mathematical evidence supporting the power law hypothesis is surprisingly scarce. As one of the few examples, \cite{Avrachenkov06gn} obtains the power law behavior of average PageRank scores in a preferential attachment graph by using Polya's urn scheme and advanced numerical methods.

In a series of papers, Volkovich et al.~\cite{Lit_Sch_Volk_07,Volk_Litv_Dona_07,Volk_Litv_10} suggested an analytical explanation for the power law behavior of PageRank by comparing the PageRank of a randomly chosen node to the endogenous solution of a stochastic fixed point equation (SFPE) that mimics (\ref{eq:PR_def2}):
\begin{equation}
\label{eq:SFPE0}
R\stackrel{{\cal D}}{=}\sum_{i=1}^NC_iR_i+Q.\end{equation}
 Here $N$ (in-degree) is a nonnegative integer random variable having a power law distribution with exponent $\alpha$, $Q$ (personalization) is an arbitrary positive random variable, and the $C_i$'s are random coefficients that in \cite{Volk_Litv_10} equal $c/D_i$, with $D_i$ being the out-degree of a node provided $D_i \geq 1$.  The symbol $\stackrel{\mathcal{D}}{=}$ denotes equality in distribution. Assuming that $N$ is regularly varying and using Laplace transforms, it was proved in \cite{Volk_Litv_10} that $R$ has a power law with the same exponent as $N$ if $N$ has a heavier tail than $Q$, whereas the tail of $R$ is determined by $Q$ if it is heavier than $N$. The same result was also proved independently in \cite{Jel_Olv_10} using a sample-path approach.

The properties of equation (\ref{eq:SFPE0}) and the study of its multiple solutions has itself been an interesting topic in the recent literature \cite{Alsm_Mein_10b, Jel_Olv_10, Jel_Olv_12a, Jel_Olv_12b, Olvera_12a, Als_Dam_Men_12}, and is related to the broader study of weighted branching processes (WBPs) \cite{Rosler_93, Ros_Top_Vat_00, Ros_Top_Vat_02}. The tail behavior of the endogenous solution, the one relevant to PageRank, was given in \cite{Jel_Olv_10, Jel_Olv_12a, Jel_Olv_12b, Olvera_12a}.  In particular, in \cite{Jel_Olv_10} it was discovered that when the $C_i$'s are not bounded by one and there exists a positive root to the equation $E\left[ \sum_{i=1}^N |C_i|^\alpha \right] = 1$ with $0 < E\left[ \sum_{i=1}^N |C_i|^\alpha \log |C_i| \right] < \infty$, then $R$ will have a power law tail with exponent $\alpha$; the main tool for this type of analysis is the implicit renewal theory on trees developed there and later extended in \cite{Jel_Olv_12a, Jel_Olv_12b} to study \eqref{eq:SFPE0} in its full generality.

However, the SFPE does not fully explain the behavior of PageRank in networks since it implicitly assumes that the underlying graph is an infinite tree, a condition that is never true in real-world networks. In this work we complete the argument when the underlying network is a Directed Configuration Model by showing that the distribution of the PageRank in the graph converges to the endogenous solution of a SFPE. Our techniques are likely to be useful in the analysis of PageRank in other locally tree-like graphs.

The essential theoretical contribution of this work is two-fold. First, we prove that the PageRank in the Directed Configuration Model is well approximated by the endogenous solution to a specific SFPE of the same type as  (\ref{eq:SFPE0}). Second, we develop a methodology to analyze processes on graphs based on a coupling with a new type of stochastic process: a {\it weighted} branching process. Due to the presence of weights, couplings with weighted branching processes are more complex compared to traditional couplings with standard branching processes, and therefore, our approach may be of independent interest.

In Section~\ref{sec:overview} we describe our main results, outline the methodology, and provide an overview of the rest of the paper.

\section{Overview of the paper}
\label{sec:overview}

Although a rigorous presentation of the main result in the paper requires a significant amount of notation, we provide here a somewhat imprecise version that still captures the essence of our work. The paper is written according to the different steps needed in the proof of the main result, outlined in Section \ref{S.Methodology}, and the precise statement can be found in Section \ref{S.MainResult}.

\subsection{An overview of the main result}

Let ${\cal G}_n =(V_n,E_n)$ be a directed graph. We number the nodes $V_n=\{1,2,\ldots,n\}$ in an arbitrary fashion and let $R_1=:R^{(n)}_1$ denote the PageRank of node~1, as defined by (\ref{eq:PR_def2}). The in-degree of node~1 is then a random variable $N_1$ picked uniformly at random from the in-degrees of all $n$ nodes in the graph (i.e., from the empirical distribution). Next, we use the notation $N_{i+1}$ to denote the in-degree of the $i$th inbound neighbor of node~1 (i.e., $(i+1,1) \in E_n$), and note that although the $\{N_i\}_{i \geq 2}$ have the same distribution, it is not necessarily the same of $N_1$ since their corresponding nodes implicitly have one or more out-degrees. More precisely, the distribution of the $\{N_i\}_{i\geq 2}$ is an empirical {\it size-biased} distribution where nodes with high out-degrees are more likely to be chosen. The two distributions can be significantly different when the number of dangling nodes (nodes with zero out-degrees) is a positive fraction of $n$ and their in-degree distribution is different than that of nodes with one or more out-degrees. Similarly, let $Q_1$ and $\{Q_i\}_{i\geq 2}$ denote the personalization values of node~1 and of its neighbors, respectively, and let $\{C_i\}_{i \geq 2}$ denote the coefficients, or weights, of the neighbors.

As already mentioned, we will assume throughout the paper that $\mathcal{G}_n$ is constructed according to the Directed Configuration Model (DCM). To briefly explain the construction of the DCM consider a bi-degree sequence $({\bf N}_n, {\bf D}_n) = \{(N_i, D_i): 1 \leq i \leq n\}$ of nonnegative integers satisfying $\sum_{i=1}^n N_i = \sum_{i=1}^n D_i$. To draw the graph think of each node, say node $i$, as having $N_i$ inbound and $D_i$ outbound half-edges or stubs, then pair each of its inbound stubs with a randomly chosen outbound stub from the set of unpaired outbound stubs (see Section \ref{S.DCM} for more details). The resulting graph is in general what is called a multigraph, i.e., it can have self-loops and multiple edges in the same direction.

Our main result requires us to make some assumptions on the bi-degree sequence used to construct the DCM, as well as on the coefficients $\{C_i\}$ and the personalization values $\{Q_i\}$, which we will refer to as the extended bi-degree sequence. The first set of assumptions (see Assumption~\ref{A.Bidegree}) requires the existence of certain limits in the spirit of the weak law of large numbers, including $\frac{1}{n}\sum_{i=1}^nD_i^2$ to be bounded in probability (which essentially imposes a finite variance on the out-degrees). This first assumption will ensure the local tree-like structure of the graph. The second set of assumptions (see Assumption~\ref{A.WeakConvergence}) requires the convergence of certain empirical distributions, derived from the extended bi-degree sequence, to proper limits as the graph size goes to infinity. This type of weak convergence assumption is typical in the analysis of random graphs~\cite{HofstadRG}. We point out that the two sets of assumptions mentioned above are rather weak, and therefore our result is very general. Moreover, as an example, we provide in Section \ref{S.Example} an algorithm to generate an extended bi-degree sequence from a set of prescribed distributions that satisfies both assumptions.

To state our main result let $(\mathcal{N}_0, \mathcal{Q}_0)$ and $(\mathcal{N}, \mathcal{Q}, \mathcal{C})$ denote the weak limits of the joint random distributions of $(N_1, Q_1)$ and $(N_2, Q_2, C_2)$, respectively, as defined in Assumption \ref{A.WeakConvergence}. Let $\mathcal{R}$ denote the endogenous solution to the following SFPE:
\begin{equation}
\label{eq:SFPE}
\mathcal{R} \stackrel{\mathcal{D}}{=} \sum_{j=1}^{\mathcal{N}} \mathcal{C}_j \mathcal{R}_j + \mathcal{Q},
\end{equation}
where $\{\mathcal{R}_i\}$ are i.i.d. copies of $\mathcal{R}$, independent of $(\mathcal{N}, \mathcal{Q}, \{\mathcal{C}_i\})$, and with $\{\mathcal{C}_i\}$ i.i.d. and independent of $(\mathcal{N}, \mathcal{Q})$. Our main result establishes that under the assumptions mentioned above, we have that
\[R_1^{(n)}\Rightarrow {\cal R}^*, \qquad n \to \infty, \]
where $\Rightarrow$ denotes weak convergence and $\mathcal{R}^*$ is given by
\begin{equation} \label{eq:FinalLimit}
\mathcal{R}^* :=  \sum_{j=1}^{\mathcal{N}_0} \mathcal{C}_j \mathcal{R}_j + \mathcal{Q}_0,
\end{equation}
where the $\{\mathcal{R}_i\}$ are again i.i.d. copies of $\mathcal{R}$, independent of $(\mathcal{N}_0, \mathcal{Q}_0, \{\mathcal{C}_i\})$, and with $\{\mathcal{C}_i\}$ independent of $(\mathcal{N}_0, \mathcal{Q}_0)$.
Thus, $R_1^{(n)}$ is well approximated by a linear combination of endogenous solutions of a SFPE. Here ${\cal R}^*$ represents the PageRank of node~1, and the $\mathcal{R}_i$'s represent the PageRank of its inbound neighbors. We give more details on the explicit construction of $\mathcal{R}$ and comment on why it is called the ``endogenous" solution in Section \ref{S.Main}. Furthermore, since $\mathcal{R}$ has been thoroughly studied in the weighted branching processes literature, we can establish the power law behavior of PageRank in a wide class of DCM graphs.

\subsection{Methodology} \label{S.Methodology}

As mentioned earlier, the proof of our main result is given in several steps, each of them requiring a very different type of analysis. For the convenience of the reader, we include in this section a map of these steps.

We start in Section~\ref{S.DCM} by describing the DCM, which on its own does not require any assumptions on the bi-degree sequence.  Then, in Section~\ref{S.spectral_ranking} we define a class of ranking algorithms, of which PageRank and its various modifications are special cases. These algorithms produce a vector ${\bf R}^{(n)}$ that is a solution to a linear system of equations, where the coefficients are the {\it weights} $\{ C_i\}$ assigned to the nodes. For example, in the classical PageRank scenario, we have $C_i=c/D_i$, if $D_i\ne 0$.

The proof of the main result consists of the following three steps:
\begin{enumerate}  \renewcommand{\labelenumi}{\arabic{enumi}.}
\item {\em Finite approximation} (Section \ref{S.FinitelyMany}). Show that the class of rankings that we study can be approximated in the DCM with any given accuracy by a finite  (independent of the graph size $n$) number of matrix iterations. The DCM plays a crucial role in this step since it implies that the ranks of all the nodes in the graph have the same distribution. A uniform bound on the sequence $\{C_i D_i\}$ is required to provide a suitable rate of convergence.

\item {\em Coupling with a tree} (Section~\ref{S.CouplingWithTree}). Construct a coupling of the DCM graph and a  ``thorny branching tree" (TBT). In a TBT each node with the exception of the root has one outbound link to its parent and possibly several other unpaired outbound links. During the construction, all nodes in both the graph and the tree are also assigned a weight $C_i$. The main result in this section is the Coupling Lemma~\ref{L.CouplingBreaks}, which states that the coupling between the graph and the tree will hold for a number of generations in the tree that is logarithmic in $n$. The locally tree-like property of the DCM and our first set of assumptions (Assumption \ref{A.Bidegree}) on the bi-degree sequence are important for this step.

\item {\em Convergence to a weighted branching process} (Section~\ref{S.Main}).  Show that the rank of the root node of the TBT converges weakly to \eqref{eq:FinalLimit}. This last step requires the weak convergence of the random distributions that define the TBT in the previous step (Assumption \ref{A.WeakConvergence}).

\end{enumerate}

Finally, Section \ref{S.Example} gives an algorithm to construct an extended bi-degree sequence satisfying the two main assumptions. The technical proofs are postponed to Section \ref{S.Proofs}.

\section{The directed configuration model}
\label{S.DCM}

The Configuration Model (CM) was originally defined as an undirected graph sampled uniformly at random from the collection of graphs with a given degree sequence~\cite{Bollobas1980CM}. In order to ensure a desired degree distribution, one may generate an i.i.d. degree sequence sampled from this distribution, see \cite[Section 7.6]{HofstadRG}. In this case each node receives a random number of half-edges, or stubs, and then the stubs are paired uniformly at random. The resulting graph is, in general, a multi-graph, because two stubs of the same node may form an edge (self-loop), or a node may have two or more stubs connected to the same other node (multiple edges). There are two ways to create a simple graph. In the {\it repeated} CM, the pairing is repeated until a simple graph is obtained. This will occur with positive probability if the degrees have finite variance, see \cite[Section 7.6.]{HofstadRG}. In the {\it erased} CM self-loops and double-edges are removed. In the erased CM, the degree sequence is altered because of edge removal, but the distribution of the original degree sequence is preserved asymptotically under very general conditions, see again \cite[Section 7.6]{HofstadRG}. A literature review and discussion of the undirected CM is provided in \cite[Section 7.9]{HofstadRG}.

While the undirected CM has been thoroughly studied, a formal analysis of the Directed Configuration Model (DCM) with given in- and out-degree distributions has only been recently presented by Chen and Olvera-Cravioto~\cite{Chen_Olv_13}. The crucial difference compared to the undirected case is that now we have a {\it bi-degree} sequence, i.e., a pair of sequences of nonnegative integers determining the in- and out-degrees of the nodes. Note that the sums of the in-degrees must be equal to that of the out-degrees for one to be able to draw a graph. The difficulty and originality of the DCM  is that sums of i.i.d. in- and out-degrees will only be equal with a probability converging to zero as the size of the graph grows. To circumvent this problem, the algorithm given in \cite{Chen_Olv_13}, and included in Section~\ref{S.Example} in this paper, forces the sums to match by adding the necessary half-edges in such a way that the degree distributions are essentially unchanged.

In order to analyze the distribution of ranking scores on the DCM we also need other node attributes besides the in- and out-degrees, such as the coefficients and the personalization values. With this in mind we give the following definition.

\begin{defn}
\label{D.bidegree}
We say that the sequence $({\bf N}_n, {\bf D}_n, {\bf C}_n, {\bf Q}_n) = \{ (N_i, D_i, C_i, Q_i): 1 \leq i \leq n \}$ is an {\em extended bi-degree sequence} if for all $1 \leq i \leq n$ it satisfies $N_i, D_i \in \mathbb{N} = \{0, 1, 2, 3, \dots\}$, $Q_i, C_i \in \mathbb{R}$, and is such that
$$L_n := \sum_{i=1}^n N_i = \sum_{i=1}^n D_i.$$
In this case, we call $({\bf N}_n, {\bf D}_n)$ a {\em bi-degree sequence}.
\end{defn}

Formally, the DCM can be defined as follows.

\begin{defn}
\label{D.DCM}
Let $({\bf N}_n, {\bf D}_n)$ be a bi-degree sequence and let $V_n = \{1, 2, \dots, n\}$ denote the nodes in the graph. To each node $i$ assign $N_i$ inbound half-edges and $D_i$ outbound half-edges. Enumerate all $L_n$ inbound half-edges, respectively outbound half-edges, with the numbers $\{1, 2, \dots, L_n\}$, and let ${\bf x}_n = (x_1, x_2, \dots, x_{L_n})$ be a random permutation of these $L_n$ numbers, chosen uniformly at random from the possible $L_n!$ permutations. The DCM with bi-degree sequence $({\bf N}_n, {\bf D}_n)$ is the directed graph $\mathcal{G}_n = (V_n, E_n)$ obtained by pairing the $x_i$th outbound half-edge with the $i$th inbound half-edge.
\end{defn}

We point out that instead of generating the permutation ${\bf x}_n$ of the outbound half-edges up front, one could alternatively construct the graph in a breadth-first fashion, by pairing each of the inbound half-edges, one at a time, with an outbound half-edge, randomly chosen with equal probability from the set of unpaired outbound half-edges. In Section \ref{S.CouplingWithTree} we will follow this approach while simultaneously constructing a coupled TBT.

We emphasize that the DCM is, in general, a multi-graph. It was shown in~\cite{Chen_Olv_13} that the random pairing of inbound and outbound half-edges results in a simple graph with positive probability provided both the in-degree and out-degree distributions possess a finite variance. In this case, one can obtain a simple realization after finitely many attempts, a method we refer to as the {\it repeated} DCM, and this realization will be chosen uniformly at random from all simple directed graphs with the given bi-degree sequence. Furthermore, if the self-loops and multiple edges in the same direction are simply removed, a model we refer to as the {\it erased} DCM, the degree distributions will remain asymptotically unchanged.

For the purposes of this paper, self-loops and multiple edges in the same direction do not affect the main convergence result for the ranking scores, and therefore we do not require the DCM to result in a simple graph. A similar observation was made in the paper by van der Hofstad et al.~\cite{Hof_Hoo_Van_05} when analyzing distances in the undirected CM.

Throughout the paper, we will use $\mathscr{F}_n = \sigma( ({\bf N}_n, {\bf D}_n, {\bf C}_n, {\bf Q}_n))$ to denote the sigma algebra generated by the extended bi-degree sequence, which does not include information about the random pairing. To simplify the notation, we will use $\mathbb{P}_n( \cdot ) = P( \cdot | \mathscr{F}_n)$ and $\mathbb{E}_n[ \cdot ] = E[ \cdot | \mathscr{F}_n]$ to denote the conditional probability and conditional expectation, respectively, given $\mathscr{F}_n$.

\section{Spectral ranking algorithms}
\label{S.spectral_ranking}

In this section we introduce the class of ranking algorithms that we analyze in this paper. Following the terminology from \cite{Boldi2014axioms}, these algorithms belong to the class of {\it spectral centrality measures}, which `compute the left dominant eigenvector of some matrix derived from the graph'.  We point out that the construction of the matrix of weights and the definition of the rank vector that we give in Section \ref{S.PR_def} is not particular to the DCM.

\subsection{Definition of the rank vector}
\label{S.PR_def}

The general class of spectral ranking algorithms we consider are determined by a matrix of weights $M = M(n) \in \mathbb{R}^{n\times n}$  and a personalization vector ${\bf Q} \in \mathbb{R}^n$. More precisely, given a directed graph with $({\bf N}_n, {\bf D}_n, {\bf C}_n, {\bf Q}_n)$ as its extended bi-degree sequence, we define the $(i,j)$th component of matrix $M$ as follows:
\begin{equation}\label{eq:M}
M_{i,j} = \begin{cases}
s_{ij} C_{i}, & \text{ if there are $s_{ij}$ edges from $i$ to $j$}, \\
0, & \text{ otherwise.}
\end{cases}\end{equation}

The rank vector ${\bf R} = (R_1, \dots, R_n)$ is then defined to be the solution to the system of equations
\begin{equation} \label{eq:LinearSystem}
{\bf R} = {\bf R} M + {\bf Q}.
\end{equation}

\begin{remark}
In the case of the PageRank algorithm, $C_i = c/D_i$, $Q_i = 1-c$ for all $i$, and the constant $0 < c < 1$ is the so-called damping factor.
\end{remark}

\subsection{Finitely many iterations} \label{S.FinitelyMany}

To solve the system of equations given in \eqref{eq:LinearSystem} we proceed via matrix iterations~\cite{Langville2011google}.  To initialize the process let ${\bf 1}$ be the (row) vector of ones in $\mathbb{R}^n$ and let ${\bf r}_0 = r_0 {\bf 1}$, with $r_0 \in \mathbb{R}$. Define
$${\bf R}^{(n,0)} = {\bf r}_0,$$
and for $k \geq 1$,
$${\bf R}^{(n,k)} = {\bf r}_0 M^k + \sum_{i=0}^{k-1} {\bf Q} M^i.$$
With this notation, we have that the solution ${\bf R}$ to \eqref{eq:LinearSystem}, provided it exists, can be written as
$${\bf R} = {\bf R}^{(n,\infty)}  = \sum_{i=0}^\infty {\bf Q} M^i.$$

We are interested in analyzing a randomly chosen coordinate of the vector ${\bf R}^{(n,\infty)}$. The first step, as described in Section~\ref{S.Methodology}, is to show that we can do so by using only finitely many matrix iterations. To this end note that
$${\bf R}^{(n,k)} - {\bf R}^{(n,\infty)} = {\bf r}_0 M^k - \sum_{i=k}^\infty {\bf Q} M^i = \left( {\bf r}_0 - \sum_{i=0}^\infty {\bf Q} M^i \right) M^k.$$
Moreover,
$$\left|\left| {\bf R}^{(n,k)} - {\bf R}^{(n,\infty)} \right|\right|_1 \leq \left|\left| {\bf r}_0 M^k \right|\right|_1 + \sum_{i=0}^\infty \left|\left| {\bf Q} M^{k+i} \right|\right|_1.$$

Next, note that for any row vector ${\bf y} = (y_1, y_2, \dots, y_n)$,
\begin{align*}
\left|\left| {\bf y} M^r \right|\right|_1 &\leq \sum_{j=1}^n \left|  {\bf y} (M^r)_{\bullet j} \right| \leq \sum_{j=1}^n \sum_{i=1}^n \left|y_i (M^r)_{ij} \right| \\
&= \sum_{i=1}^n |y_i| \sum_{j=1}^n |(M^r)_{ij}| = \sum_{i=1}^n |y_i| \cdot || M^r_{i\bullet} ||_1 \\
&\leq || {\bf y} ||_1 \left|\left| M^r \right|\right|_\infty,
\end{align*}
where $A_{i \bullet}$ and $A_{\bullet j}$ are the $i$th row and $j$th column, respectively, of matrix $A$, and $|| A||_\infty = \max_{1 \leq i \leq n} ||A_{i \bullet}||_1$ is the operator infinity norm. It follows that if we assume that $\max_{1 \leq i \leq n} |C_i| D_i \leq c$ for some $c \in (0,1)$, then we have
$$\left|\left| M^r \right|\right|_\infty \leq || M ||_\infty^r = \left( \max_{1 \leq i \leq n} |C_i| D_i \right)^r \leq c^r.$$
In this case we conclude that
\begin{align*}
\left|\left| {\bf R}^{(n,k)} - {\bf R}^{(n,\infty)} \right|\right|_1 &\leq || {\bf r}_0 ||_1 c^k + \sum_{i=0}^\infty || {\bf Q} ||_1c^{k+i} \\
&= |r_0| n c^k + || {\bf Q} ||_1 \frac{c^k}{1 - c}.
\end{align*}

Now note that all the coordinates of the vector ${\bf R}^{(n,k)} - {\bf R}^{(n,\infty)}$ have the same distribution, since by construction, the configuration model makes all permutations of the nodes' labels equally likely. Hence, the randomly chosen node may as well be the first node, and the error that we make by considering only finitely many iterations in its approximation is bounded in expectation by
\begin{align*}
\mathbb{E}_n \left[  \left| R^{(n,k)}_1 - R^{(n,\infty)}_1 \right| \right] &= \frac{1}{n} \mathbb{E}_n \left[ \left|\left| {\bf R}^{(n,k)} - {\bf R}^{(n,\infty)} \right|\right|_1 \right] \\
&\leq |r_0| c^k + \mathbb{E}_n \left[ || {\bf Q} ||_1  \right] \frac{c^k}{n(1-c)} \\
&= \left( |r_0| + \frac{1}{n(1-c)}\sum_{i=1}^n |Q_i| \right) c^k.
\end{align*}

It follows that if we let
\begin{equation} \label{eq:WeakestConditions}
B_n = \left\{ \max_{1 \leq i \leq n} |C_i| D_i \leq c, \, \frac{1}{n} \sum_{i=1}^n |Q_i| \leq H \right\}
\end{equation}
for some constants $c \in (0,1)$ and $H < \infty$, then Markov's inequality yields
\begin{align}
&P\left( \left. \left| R^{(n,k)}_1 - R^{(n,\infty)}_1\right| > n^{-\epsilon}  \right| B_n \right) \notag  \\
&= \frac{1}{P(B_n)} E\left[ 1(B_n) \mathbb{E}_n \left[ 1\left(  \left| R^{(n,k)}_1 - R^{(n,\infty)}_1\right| > n^{-\epsilon} \right)  \right] \right] \notag \\
&\leq \frac{1}{P(B_n)} E\left[ 1(B_n) n^\epsilon \mathbb{E}_n\left[ \left| R^{(n,k)}_1 - R^{(n,\infty)}_1 \right|  \right] \right] \notag \\
&\leq \left( |r_0| + \frac{1}{1-c} E\left[ \left. \frac{1}{n} \sum_{i=1}^n |Q_i| \right| B_n \right] \right) n^\epsilon c^k \notag \\
&\leq \left( |r_0| + \frac{H}{1-c} \right) n^\epsilon c^k . \label{eq:PowerIterations}
\end{align}

We have thus derived the following result.

\begin{prop} \label{P.PowerIterations}
Consider the directed configuration graph generated by the extended bi-degree sequence $({\bf N}_n, {\bf D}_n, {\bf C}_n, {\bf Q}_n)$ and let $B_n$ be defined according to \eqref{eq:WeakestConditions}. Then, for any $x_n \to \infty$ and any $k \geq 1$, we have
$$P\left( \left. \left| R^{(n,\infty)}_1 - R^{(n,k)}_1\right| > x_n^{-1}  \right| B_n \right) = O\left( x_n c^k \right)$$
as $n \to \infty$.
\end{prop}

This completes the first step of our approach. In the next section we will explain how to couple the graph, as seen from a randomly chosen node, with an appropriate branching tree.

\section{Construction of the graph and coupling with a branching tree} \label{S.CouplingWithTree}

The next step in our approach is to approximate the distribution of $R_1^{(n,k)}$ with the rank of the root node of a suitably constructed branching tree. To ensure that we can construct such a tree we require the extended bi-degree sequence to satisfy some further properties with high probability.  These properties are summarized in the following assumption.

\begin{assum} \label{A.Bidegree}
Let $({\bf N}_n, {\bf D}_n, {\bf C}_n, {\bf Q}_n)$ be an extended bi-degree sequence for which there exists constants $H, \nu_i > 0$, $i = 1, \dots, 5$, with
$$\mu := \nu_2/\nu_1, \quad \lambda := \nu_3/\nu_1  \qquad \text{and} \qquad \rho := \nu_5 \mu/\nu_1 < 1, $$
$0 < \kappa \leq 1$, and $0 < c, \gamma, \epsilon < 1$ such that the events
\begin{align*}
\Omega_{n,1} &=  \left\{  \left| \sum_{r=1}^n D_{r} - n \nu_1 \right| \leq n^{1-\gamma} \right\}, \\
\Omega_{n,2} &= \left\{ \left| \sum_{r=1}^n D_{r} N_{r} - n \nu_2 \right| \leq n^{1-\gamma} \right\}, \\
\Omega_{n,3} &= \left\{  \left| \sum_{r=1}^n D^2_{r} - n \nu_3 \right| \leq n^{1-\gamma} \right\}, \\
\Omega_{n,4} &= \left\{ \left| \sum_{r=1}^n D_{r}^{2+\kappa}  - n \nu_4 \right| \leq n^{1-\gamma} \right\}, \\
\Omega_{n,5} &= \left\{ \left| \sum_{r=1}^n |C_r| D_r  - n \nu_5 \right| \leq n^{1-\gamma}, \, \max_{1\leq r \leq n} |C_r| D_r \leq c \right\}, \\
\Omega_{n,6} &= \left\{  \sum_{r=1}^n |Q_r| \leq H n \right\}, \\
\end{align*}
satisfy as $n \to \infty$,
$$P\left( \Omega_n^c \right) = P\left( \left( \bigcap_{i=1}^6 \Omega_{n,i} \right)^c \right) = O\left(  n^{-\epsilon} \right).$$
\end{assum}

It is clear from \eqref{eq:WeakestConditions} that $\Omega_n \subseteq B_n$, hence Proposition~\ref{P.PowerIterations} holds under Assumption~\ref{A.Bidegree}. We also point out that all six conditions in the assumption are in the spirit of the Weak Law of Large Numbers, and are therefore general enough to be satisfied by many different constructions of the extended bi-degree sequence. As an example, we give in Section~\ref{S.Example} an algorithm based on sequences of i.i.d. random variables that satisfies Assumption~\ref{A.Bidegree}.

In Sections~\ref{S.Coupling_Terminology}--\ref{S.Ranking_TBT} we describe in detail how to construct a coupling of the directed graph $\mathcal{G}_n$ and its approximating branching tree. We start by explaining the terminology and notation in Section~\ref{S.Coupling_Terminology}, followed by the construction itself in Section \ref{S.Coupling_Construction}. Then, in Section \ref{S.Coupling_Lemma} we present the Coupling Lemma~\ref{L.CouplingBreaks}, which is the main result of Section \ref{S.CouplingWithTree}. Finally, Section \ref{S.Ranking_TBT} explains how to compute the rank of the root node in the coupled tree.

\subsection{Terminology and notation}
\label{S.Coupling_Terminology}

Throughout the remainder of the paper we will interchangeably refer to the $\{N_i\}$ as the in-degrees/number of offspring/number of inbound stubs, to the $\{D_i\}$ as the out-degrees/number of outbound links/number of outbound stubs, to the $\{C_i\}$ as the weights, and to the $\{Q_i\}$ as the personalization values. We will refer to these four characteristics of a node as the {\em node attributes}.

The fact that we are working with a directed graph combined with the presence of weights, means that we need to use a more general kind of tree in our coupling than the standard branching process typically used in the random graph literature. To this end, we will define a process we call a Thorny Branching Tree (TBT), where each individual (node) in the tree has a directed edge pointing towards its parent, and also a certain number of unpaired outbound links (pointing, say, to an artificial node outside of the tree). The name `thorny' is due to these unpaired outbound links, see Figure \ref{F.Tree}. We point out that the structure of the tree (i.e., parent-offspring relations) is solely determined by the number of offspring.

\begin{figure}[h,t]
\centering
\includegraphics[scale = 0.6, bb = 70 430 570 630, clip]{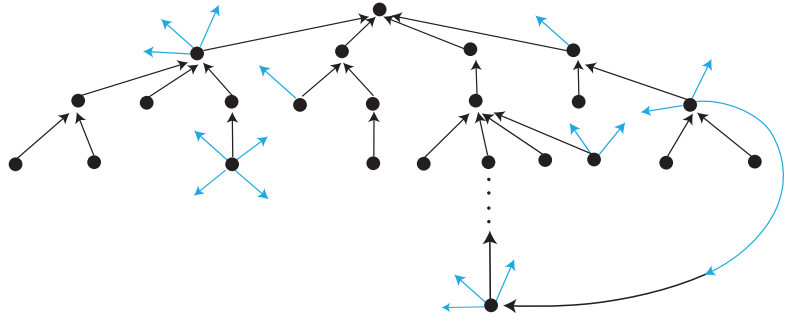}
\caption{Graph construction process. Unpaired outbound links are in blue.}\label{F.Tree}
\end{figure}

The simpler structure of a tree compared to a general graph allows for a more natural enumeration  of  its nodes. As usually in the context of branching processes, we let each node in the TBT have a label of the form ${\bf i} = (i_1, i_2, \dots, i_k) \in \mathcal{U}$, where $\mathcal{U} =  \bigcup_{k=0}^\infty (\mathbb{N}_+)^k$ is the set of all finite sequences of positive integers. Here, the convention is that $\mathbb{N}_+^0 = \{ \emptyset \}$ contains the null sequence $\emptyset$. Also, for ${\bf i}=(i_1)$ we simply write ${\bf i} = i_1$, that is, without the parenthesis.
Note that this form of enumeration gives the complete lineage of each individual in the tree.

We will use the following  terminology and notation throughout the paper.

\begin{defn}
We say that a node $i$ in the graph (resp. TBT) is at distance $k$ of the first (resp. root) node if it can reach the first (resp. root) node in $k$ steps, but not in any less than $k$ steps.
\end{defn}

In addition, for $r \geq 0$, we define on the graph/tree the following processes:
\begin{itemize} \itemsep 0pt
\item $A_r$: set of nodes in the graph at distance $r$ of the first node.

\item $\hat A_r$: set of nodes in the tree at distance $r$ of the root node ($\hat A_r$ is also the set of nodes in the $r$th generation of TBT, with the root node being generation zero).

\item $Z_r$: number of inbound stubs of all the nodes in the graph at distance $r$ of the first node ($Z_r \geq |A_{r+1}|$).

\item $\hat Z_r$: number of inbound stubs of all the nodes in generation $r$ of the TBT ($\hat Z_r = |\hat A_{r+1}|$).

\item $V_r$: number of outbound stubs of all the nodes in the graph at distance $r$ of the first node.

\item $\hat V_r$: number of outbound stubs of all the nodes in generation $r$ of the TBT.

\end{itemize}

Finally, given the extended bi-degree sequence $({\bf N}_n, {\bf D}_n, {\bf C}_n, {\bf Q}_n)$, we introduce two empirical distributions that will be used in the construction of the coupling. The first one describes the attributes of a randomly chosen node:
\begin{align}
f_{n}^*(i,j, s, t)
&= \sum_{k=1}^n \Indicator(N_k = i, D_k = j, C_k = s, Q_k = t) \mathbb{P}_n(\text{node $k$ is sampled}) \notag \\
&= \frac{1}{n} \sum_{k=1}^n \Indicator(N_k = i, D_k = j, C_k = s, Q_k = t). \label{eq:FirstNodeDistr}
\end{align}

The second one, corresponds to the attributes of a node that is chosen by sampling uniformly at random from all the $L_n$ outbound stubs:
\begin{align}
f_n(i,j,s,t)
&= \sum_{k=1}^n \Indicator(N_k = i, D_k = j, C_k = s, Q_k = t) \mathbb{P}_n(\text{an outbound stub from node $k$ is sampled}) \notag \\
&=  \sum_{k=1}^n \Indicator(N_k = i, D_k = j, C_k = s, Q_k = t) \frac{D_k}{L_n}. \label{eq:randomJointDistr}
\end{align}
Note that this is a size-biased distribution, since nodes with more outbound stubs are more likely to be chosen, whereas nodes with no outbound stubs (dangling nodes) cannot be chosen.

\subsection{Construction of the coupling}
\label{S.Coupling_Construction}

Given an extended bi-degree sequence $({\bf N}_n, {\bf D}_n, {\bf C}_n, {\bf Q}_n)$ we now explain how to construct the graph $\mathcal{G}_n$ and its coupled TBT through a breadth-first exploration process. From this point onwards we will ignore the implicit numbering of the nodes in the definition of the extended bi-degree sequence and rename them according to the order in which they appear in the graph exploration process.

To keep track of which outbound stubs have already been matched we borrow the approach used in \cite{Hof_Hoo_Van_05} and label them 1, 2, or 3 according to the following rules:
\begin{list}{}{ \itemsep 0pt \leftmargin 5mm}
\item [1.] Outbound stubs with label 1 are stubs belonging to a node that is not yet attached to the graph.
\item [2.] Outbound stubs with label 2 belong to nodes that are already part of the graph but that have not yet been paired with an inbound stub.
\item [3.] Outbound stubs with label 3 are those which have already been paired with an inbound stub and now form an edge in the graph.
\end{list}

The graph $\mathcal{G}_n$ is constructed as follows. Right before the first node is sampled, all outbound stubs are labeled 1.  To start the construction of the graph, we choose randomly a node (all nodes with the same probability) and call it node 1. The attributes of this first node, denoted by $(N_1,D_1,C_1,Q_1)$, are sampled from distribution $(\ref{eq:FirstNodeDistr})$.

After the first node is chosen, its $D_1$ outbound stubs are labeled 2. We then proceed to pair the first of the $Z_0=N_1$ inbound stubs of the first node with a randomly chosen outbound stub. The corresponding node is attached to the graph by forming an edge pointing to node 1 using the chosen outbound stub, which receives a label 3, and all the remaining outbound stubs from the new node are labeled 2. Note that it is possible that the chosen node is node 1 itself, in which case the pairing forms a self-loop and no new nodes are added to the graph. We continue in this way until all $Z_0$ inbound stubs of node 1 have been paired with randomly chosen outbound stubs. Since these outbound stubs are sampled independently and with replacement from all the possible $L_n$ outbound stubs, this corresponds to drawing the node attributes independently from the random distribution \eqref{eq:randomJointDistr}. Note that in the construction of the graph any unfeasible matches will be discarded, and therefore the attributes of nodes in $\mathcal{G}_n$ do not necessarily have distribution \eqref{eq:randomJointDistr}, but rather have the conditional distribution given the pairing was feasible. We will use the vector $(N_i, D_i, C_i, Q_i)$ to denote the attributes of the $i$th node to be added to the graph.

In general, the $k$th iteration of this process is completed when all $Z_{k-1}$ inbound stubs have been matched with an outbound stub, and the corresponding node attributes have been assigned. The process ends when all $L_n$ inbound stubs have been paired. Note that whenever an outbound stub with label 2 is chosen a cycle or a double edge is formed in the graph.

Next, we explain how the TBT is constructed. To distinguish the attribute vectors of nodes in the TBT from those of nodes in the graph, we denote them by $(\hat N_{\bf i}, \hat D_{\bf i}, \hat C_{\bf i}, \hat Q_{\bf i})$, ${\bf i} \in \mathcal{U}$. We start with the root node (node $\emptyset$) that has the same attributes  as node~1 in the graph: $(\hat N_\emptyset,\hat D_\emptyset,\hat C_\emptyset,\hat Q_\emptyset)\equiv (N_1,D_1,C_1,Q_1)$, sampled from distribution \eqref{eq:FirstNodeDistr}. Next, for $k \geq 1$, each of the $\hat Z_{k-1}$ individuals in the $k$th generation will independently have offspring, outbound stubs, weight and personalization value according to the joint distribution $f_n(i,j,s,t)$ given by \eqref{eq:randomJointDistr}.

Now, we explain how the coupling with the graph, i.e., the simultaneous construction of the graph and the TBT, is done.
\begin{itemize}
\item[1)] Whenever an outbound stub is sampled randomly in an attempt to add an edge to $\mathcal{G}_n$, then, independently of the stub's label, a new offspring is added to the TBT. This is done to maintain the branching property (i.i.d. node attributes). In particular, if the chosen outbound stub belongs to node $j$, then the new offspring in the TBT will have $D_j - 1$ outbound stubs (which will remain unpaired), $N_j$ inbound stubs (number of offspring), weight $C_j$, and personalization value $Q_j$.

\item[2)] If an outbound stub with label 1 is chosen, then both the graph and the TBT will connect the chosen outbound stub to the inbound stub being matched, resulting in a node being added to the graph and an offspring being born to its parent. We then update the labels by giving a 2 label to all the `sibling' outbound stubs of the chosen outbound stub, and a 3 label to the chosen outbound stub itself.

\item[3)] If an outbound stub with label 2 is chosen it means that its corresponding node already belongs to the graph, and a cycle, self-loop, or multiple edge is created. We then relabel the chosen outbound stub with a 3. An offspring is born in the TBT according to 1).

\item[4)] If an outbound stub with label 3 is chosen it means that the chosen outbound stub has already been matched. In terms of the construction of the graph, this case represents a failed attempt to match the current inbound stub, and we have to keep sampling until we draw an outbound stub with label 1 or 2.  Once we do so, we update the labels according to the rules given above. An offspring is born in the TBT according to 1).
\end{itemize}

Note that as long as we do not sample any outbound stub with label 2 or 3, the graph $\mathcal{G}_n$ and the TBT are identical. Once we draw the first outbound stub with label 2 or 3 the processes $Z_k$ and $\hat Z_k$ may start to disagree. The moment this occurs we say that the coupling has been broken. Nonetheless, we will continue with the pairing process following the rules given above until all $L_n$ inbound stubs have been paired. The construction of the TBT also continues in parallel by keeping the synchronization of the pairing whenever the inbound stub being matched belongs to a node that is both in the graph and the tree. If the pairing of all $L_n$ inbound stubs is completed after $k$ iterations of the process, then we will have completed $k$ generations in the TBT.
Moreover, up to the time the coupling breaks, a node ${\bf i} \in \hat A_k$ is also the $j$th node to be added to the graph,  where:
$$j = 1 + \sum_{r=0}^{k-2} \hat Z_r + \sum_{s=1}^{i_{k-1} - 1} \hat N_{(i_1, \dots, i_{k-2}, s)} + i_k,$$
with the convention that $\sum_{r=a}^b x_r = 0$ if $b < a$.

\begin{defn}
Let $\tau$ be the number of generations in the TBT that can be completed before the first outbound stub with label 2 or 3 is drawn, i.e., $\tau = k$ if and only if the first inbound stub to draw an outbound stub with label 2 or 3 belonged to a node ${\bf i} \in \hat A_k$.
\end{defn}

The main result in this section consists in showing that provided the extended bi-degree sequence $({\bf N}_n, {\bf D}_n, {\bf C}_n, {\bf Q}_n)$ satisfies Assumption \ref{A.Bidegree}, the coupling breaks only after a number of generations that is of order $\log n$, which combined with Proposition \ref{P.PowerIterations} will allow us to approximate the rank of a randomly chosen node in the graph with the rank of the root node of the coupled TBT.

\subsection{The coupling lemma}
\label{S.Coupling_Lemma}

It follows from the construction in Section~\ref{S.Coupling_Construction} that, before the coupling breaks, the neighborhood of node~1 in $\mathcal{G}_n$ and of the root node in the TBT are identical. Recall also from Proposition~\ref{P.PowerIterations} that we only need a finite number $k$ of matrix iterations to approximate the elements of the rank vector to any desired precision. Furthermore, the weight matrix $M$ is such that the elements $(M^r)_{i,1}$, $1 \leq i \leq n$, $1 \leq r \leq k$, depend only on the $k$-neighborhood of node 1. Hence, if the coupling holds for $\tau>k$ generations, then the rank score of node~1 in $\mathcal{G}_n$ is exactly the same as that of the root node of the TBT restricted to those same $k$ generations. The following coupling lemma will allow us to complete the appropriate number of generations in the tree to obtain the desired level of precision in Proposition~\ref{P.PowerIterations}.  Its proof is rather technical and is therefore postponed to Section~\ref{S.ProofCoupling}.

\begin{lemma} \label{L.CouplingBreaks}
Suppose $({\bf N}_n, {\bf D}_n, {\bf C}_n, {\bf Q}_n)$ satisfies Assumption \ref{A.Bidegree}. Then,
\begin{itemize}
\item for any $1 \leq k \leq h \log n$ with $0 < h < 1/(2\log \mu)$, if $\mu > 1$,
\item for any $1 \leq k \leq n^b$ with $0 < b < \min\{ 1/2, \gamma\}$, if $\mu \leq 1$,
\end{itemize}
we have
$$P\left( \left. \tau \leq k \right| \Omega_n \right) = \begin{cases} O\left( (n/\mu^{2k})^{-1/2} \right), & \mu > 1,\\
O\left( (n/k^2)^{-1/2} \right), & \mu = 1,\\
O\left( n^{-1/2} \right), & \mu < 1, \end{cases}$$
as $n \to \infty$.
\end{lemma}

\begin{remark}
The constant $\mu $ was defined in Assumption \ref{A.Bidegree}, and it corresponds to the limiting expected number of offspring that each node in the TBT (with the exception of the root node) will have. The coupling between the graph and the TBT will hold for any $\mu > 0$.
\end{remark}

We conclude from Lemma~\ref{L.CouplingBreaks} that if $\hat R^{(n,k)} := \hat R^{(n,k)}_{\emptyset}$ denotes the rank of the root node of the TBT restricted to the first $k$ generations, then, for any $\delta > 0$,
$$P\left( \left. \left| R_1^{(n,k)} - \hat R^{(n,k)} \right| > n^{-\delta} \right| \Omega_n \right) \leq P( \tau < k | \Omega_n) := \varphi(k,n).$$
Note that the super index $n$ does not refer to the number of nodes in the tree, and is being used only in the definition of the distributions $f_n^*$ and $f_n$ (given in \eqref{eq:FirstNodeDistr} and \eqref{eq:randomJointDistr}, respectively).

This observation, combined with Proposition \ref{P.PowerIterations}, implies that if we let $k_n = \lceil h \log n \rceil$, when $\mu > 1$, and $k_n = n^\varepsilon$, when $\mu \leq 1$, where $h = (1-\varepsilon)/(2\log \mu)$ and $0 < \varepsilon < \min\{1/3, \gamma\}$, then
\begin{align}
P\left( \left. \left| R_1^{(n,\infty)} - \hat R^{(n,k_n)} \right| > n^{-\delta} \right| \Omega_n \right) &\leq P\left( \left. \left| R_1^{(n,\infty)} - R_1^{(n,k_n)} \right| > n^{-\delta}/2 \right| \Omega_n \right) \notag \\
&\hspace{5mm} + P\left( \left. \left| R_1^{(n,k_n)} - \hat R^{(n,k_n)} \right| > n^{-\delta}/2 \right| \Omega_n \right) \notag \\
&= O\left( n^\delta c^{k_n} + \varphi(k_n, n) \right) \notag \\
&= O\left(  n^{\delta - h |\log c|} +  n^{-\varepsilon/2}   \right) \label{eq:TreeApprox}.
\end{align}

In view of \eqref{eq:TreeApprox}, analyzing the distribution of $R_1^{(n,k)}$ in the graph reduces to analyzing the rank of the root node of the coupled TBT, $\hat R^{(n,k)}$. In the next section, we compute $\hat R^{(n,k)}$ by relating it to a linear process constructed on the TBT.

\subsection{Computing the rank of nodes in the TBT}
\label{S.Ranking_TBT}

In order to compute $\hat R^{(n,k)}$ we need to introduce a new type of weights. To simplify the notation, for ${\bf i} = (i_1, \dots, i_k)$ we will use $({\bf i}, j) = (i_1,\dots, i_k, j)$ to denote the index concatenation operation; if ${\bf i} = \emptyset$, then $({\bf i}, j) = j$. Each node ${\bf i}$ is then assigned a weight $\hat \Pi_{\bf i}$ according to the recursion
$$\hat \Pi_{\emptyset} \equiv 1  \qquad \text{and} \qquad \hat \Pi_{({\bf i}, j)} = \hat \Pi_{\bf i} \hat C_{({\bf i}, j)}, \quad {\bf i} \in \mathcal{U}.$$
Note that the $\hat \Pi_{\bf i}$'s are the products of all the weights $\hat C_{\bf j}$ along the path leading to node ${\bf i}$, as depicted in Figure \ref{F.WBT}.

Next, for each fixed $k \in \mathbb{N}$ and each node ${\bf i}$ in the TBT define $\hat R^{(n,k)}_{\bf i}$ to be the rank of node ${\bf i}$ computed on the subtree that has ${\bf i}$ as its root and that is restricted to having only $k$ generations, with each of the $|\hat A_k|$ nodes having rank $r_0$. In mathematical notation,
\begin{align} \label{eq:TreeIteration}
\hat R_{\bf i}^{(n,k)} &= \sum_{j=1}^{\hat N_{\bf i}} \hat C_{({\bf i}, j)} \hat R^{(n,k-1)}_{({\bf i}, j)} + \hat Q_{\bf i}, \qquad k \geq 1, \qquad \hat R_{\bf j}^{(n,0)} = r_0.
\end{align}
Iterating \eqref{eq:TreeIteration} gives
\begin{align}
\hat R^{(n,k)} &= \sum_{{\bf i} \in \hat A_1} \hat \Pi_{\bf i} \hat R^{(n,k-1)}_{\bf i} + \hat Q_\emptyset  =   \sum_{{\bf i} \in \hat A_1} \hat \Pi_{\bf i} \left(  \sum_{j = 1}^{\hat N_{\bf i}} \hat C_{({\bf i},  j)} \hat R^{(n,k-2)}_{({\bf i}, j)} + \hat Q_{\bf i}  \right) + \hat  Q_\emptyset \notag \\
&= \sum_{{\bf i} \in \hat A_2} \hat \Pi_{\bf i} \hat R_{\bf i}^{(n,k-2)} + \sum_{{\bf i} \in \hat A_1} \hat \Pi_{\bf i} \hat Q_{\bf i} + \hat Q_\emptyset = \cdots = \sum_{{\bf i} \in \hat A_k} \hat \Pi_{\bf i} r_0 + \sum_{s=0}^{k-1} \sum_{{\bf i} \in \hat A_s} \hat \Pi_{\bf i} \hat Q_{\bf i} . \label{eq:WBPrepresentation}
\end{align}

The last step in our proof of the main result is to identify the limit of $\hat R^{(n,k_n)}$ as $n \to \infty$, for a suitable chosen $k_n\to\infty$. This is done in the next section.

\begin{center}
\begin{figure}[t]
\begin{picture}(480,110)(-40,0)
\put(0,0){\includegraphics[scale = 0.75, bb = 30 560 510 695, clip]{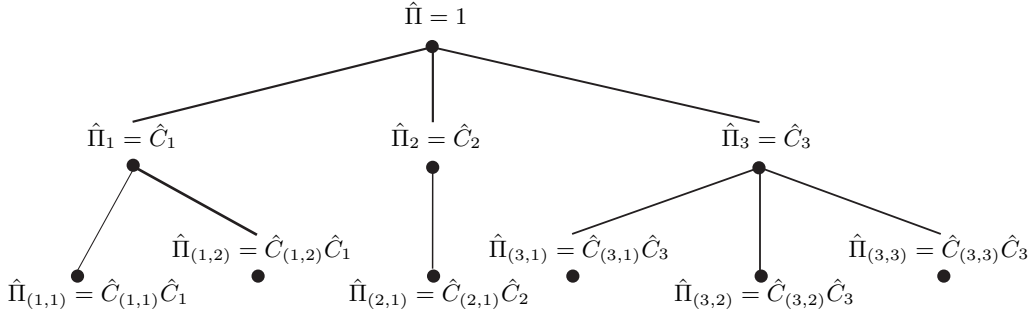}}
\put(150,105){\footnotesize $\hat \Pi = 1$}
\put(30,59){\footnotesize $\hat \Pi_{1} = \hat C_1$}
\put(145,59){\footnotesize $\hat \Pi_{2} = \hat C_2$}
\put(270,59){\footnotesize $\hat \Pi_{3} = \hat C_3$}
\put(0,0){\footnotesize $\hat \Pi_{(1,1)} = \hat C_{(1,1)} \hat C_1$}
\put(62,17){\footnotesize $\hat \Pi_{(1,2)} = \hat C_{(1,2)} \hat C_1$}
\put(129,0){\footnotesize $\hat \Pi_{(2,1)} = \hat C_{(2,1)} \hat C_2$}
\put(182,17){\footnotesize $\hat \Pi_{(3,1)} = \hat C_{(3,1)} \hat C_3$}
\put(252,0){\footnotesize $\hat \Pi_{(3,2)} = \hat C_{(3,2)} \hat C_3$}
\put(318,17){\footnotesize $\hat \Pi_{(3,3)} = \hat C_{(3,3)} \hat C_3$}
\end{picture}
\caption{Weighted tree.}\label{F.WeightedTree} \label{F.WBT}
\end{figure}
\end{center}

\section{Coupling with a weighted branching process} \label{S.Main}

The last step in the derivation of our approximation for the rank of a randomly chosen node in the graph $\mathcal{G}_n$ is to substitute the rank of the root node in the TBT, which is defined with respect to empirical distributions based on the extended bi-degree sequence $({\bf N}_n, {\bf D}_n, {\bf C}_n, {\bf Q}_n)$, with a limiting random variable independent of the size of the graph, $n$.

The appropriate limit will be given in terms of a solution to a certain stochastic fixed-point equation (SFPE). The appeal of having such a representation is that these solutions have been thoroughly studied in the WBPs literature, and in many cases exact asymptotics describing their tail behavior are available \cite{Jel_Olv_10, Jel_Olv_12b, Olvera_12a}. We will elaborate more on this point after we state our main result.

As already mentioned in Section \ref{sec:overview}, our main result shows that
$$R_1^{(n,\infty)} \Rightarrow  \mathcal{R}^*$$
as $n \to \infty$, where $\mathcal{R}^*$ can be written in terms of the so-called endogenous solution to a linear SFPE. Before we write the expression for $\mathcal{R}^*$ we will need to introduce a few additional concepts.

\subsection{The linear branching stochastic fixed-point equation}
\label{S.fixed-point}

We define the linear branching SFPE according to:
\begin{equation} \label{eq:generalSFPE}
\mathcal{R} \stackrel{\mathcal{D}}{=} \sum_{j=1}^{\mathcal{N}} \mathcal{C}_j \mathcal{R}_j + \mathcal{Q},
\end{equation}
where $(\mathcal{N}, \mathcal{Q}, \mathcal{C}_1, \mathcal{C}_2, \dots )$ is a real-valued random vector with $\mathcal{N} \in \mathbb{N} \cup \{ \infty \}$, $P( |\mathcal{Q}| > 0) > 0$, and the $\{\mathcal{R}_i\}$ are i.i.d. copies of $\mathcal{R}$, independent of the vector $(\mathcal{N}, \mathcal{Q}, \mathcal{C}_1, \mathcal{C}_2, \dots )$.  The vector $(\mathcal{N}, \mathcal{Q}, \mathcal{C}_1, \mathcal{C}_2, \dots)$ is often referred to as the generic branching vector, and in the general setting is allowed to be arbitrarily dependent with the weights $\{\mathcal{C}_i\}$ not necessarily identically distributed. This equation is also known as the ``smoothing transform" \cite{Holl_Ligg_81, Durr_Ligg_83, Als_Big_Mei_10, Alsm_Mein_10a}.

In the context of ranking algorithms, we can identify $\mathcal{N}$ with the in-degree of a node, $\mathcal{Q}$ with its personalization value, and the $\{\mathcal{C}_i\}$ with the weights of the neighboring nodes pointing to it. We now explain how to construct a solution to \eqref{eq:generalSFPE}.

Similarly as what we did in Section \ref{S.Ranking_TBT} and using the same notation introduced there, we construct a weighted tree using a sequence  $\{ (\mathcal{N}_{\bf i}, \mathcal{Q}_{\bf i}, \mathcal{C}_{({\bf i},1)}, \mathcal{C}_{({\bf i},2)}, \dots) \}_{{\bf i} \in \mathcal{U}}$ of i.i.d. copies of the vector $(\mathcal{N}, \mathcal{Q}, \mathcal{C}_1, \mathcal{C}_2, \dots)$ to define its structure and its node attributes.  This construction is known in the literature as a WBP \cite{Rosler_93}. Next, let $\mathcal{A}_k$ denote the number of individuals in the $k$th generation of the tree, and to each node {\bf i} in the tree assign a weight $\Pi_{\bf i}$ according to the recursion
$$\Pi_\emptyset \equiv 1 \qquad \text{and} \qquad \Pi_{({\bf i}, j)} = \Pi_{\bf i} \, \mathcal{C}_{({\bf i}, j)}, \quad {\bf i} \in \mathcal{U}.$$
Then, the random variable formally defined as
\begin{equation} \label{eq:EndogenousSol}
\mathcal{R} :=  \sum_{k=0}^\infty \sum_{{\bf i} \in \mathcal{A}_k} \Pi_{\bf i} \mathcal{Q}_{\bf i}
\end{equation}
is called the endogenous solution to \eqref{eq:generalSFPE}, and provided $E\left[ \sum_{i=1}^{\mathcal{N}} |\mathcal{C}_i|^\beta \right] < 1$ for some $0 < \beta \leq 1$, it is well defined (see \cite{Jel_Olv_12b}, Lemma 4.1).  The name ``endogenous" comes from its explicit construction in terms of the weighted tree. We point out that equation \eqref{eq:generalSFPE} has in general multiple solutions \cite{Alsm_Mein_10a, Alsm_Mein_10b}, so it is important to emphasize that the one considered here is the endogenous one.

Comparing \eqref{eq:WBPrepresentation} and \eqref{eq:EndogenousSol} suggests that $\hat R^{(n,k_n)}$ should converge to $\mathcal{R}$ provided the distribution of the attribute vectors in the TBT converges to the distribution of the generic branching vector in the WBP, but in order to formalize this heuristic there are two difficulties that we need to overcome. The first one is that the TBT was defined using a sequence of (conditionally) independent vectors of the form $\{ (\hat N_{\bf i}, \hat Q_{\bf i}, \hat C_{\bf i})\}_{{\bf i} \in \mathcal{U}}$, where by construction (see Assumption \ref{A.Bidegree} and \eqref{eq:randomJointDistr}) the generic attribute vector $(\hat N_1, \hat Q_1, \hat C_1)$ is dependent. Note that this implies that the vectors $(\hat N_{\bf i}, \hat Q_{\bf i}, \hat C_{({\bf i}, 1)}, \hat C_{({\bf i}, 2)}, \dots)$ and $\{(\hat N_{({\bf i},j)}, \hat Q_{({\bf i},j)}, \hat C_{({\bf i}, j,1)}, \hat C_{({\bf i}, j,2)}, \dots)\}_{j \geq 1}$ are dependent through the dependence between $\hat N_{({\bf i},j)}$ and $\hat C_{({\bf i},j)}$, which destroys the branching property of the WBP. The second problem is that the root node of the TBT has a different distribution from the rest of the nodes in the tree.

It is therefore to be expected that we will need something more than weak convergence of the node attributes to obtain the convergence of $\hat R^{(n,k_n)}$ we seek. To solve the first problem we will require that $(\hat N_1, \hat Q_1, \hat C_1)$ converges to $(\mathcal{N}, \mathcal{Q}, \mathcal{C})$ with $\mathcal{C}$ independent of  $(\mathcal{N}, \mathcal{Q})$. Note that this will naturally lead to the $\{ \mathcal{C}_i\}$ being i.i.d. in \eqref{eq:generalSFPE}.  To solve the second problem we will allow the attributes of the root node in the TBT to converge to their own limit $(\mathcal{N}_0, \mathcal{Q}_0)$. In view of these observations we can now identify the limit of $\hat R^{(n,k_n)}$ to be:
\begin{equation} \label{eq:LastLimit}
\mathcal{R}^* := \sum_{i=1}^{\mathcal{N}_0} \mathcal{C}_i \mathcal{R}_i + \mathcal{Q}_0,
\end{equation}
where the $\{\mathcal{R}_i\}$ are i.i.d. copies of $\mathcal{R}$, as given by \eqref{eq:EndogenousSol}, independent of the vector $(\mathcal{N}_0, \mathcal{Q}_0, \{ \mathcal{C}_i\})$ with $\{\mathcal{C}_i\}$ i.i.d. and independent of $(\mathcal{N}_0, \mathcal{Q}_0)$.  The appropriate condition ensuring that $\mathcal{R}^*$ is the correct limit is  given in terms of the Kantorovich-Rubinstein distance (also known as the minimal $l_1$ distance or the Wasserstein distance).

\begin{defn} \label{D.Wasserstein}
Consider the metric space $(\mathbb{R}^d, || \cdot ||_1)$, where $||{\bf x}||_1$ is the $l_1$ norm in $\mathbb{R}^d$. Let $M(\mu, \nu)$ denote the set of joint probability measures on $\mathbb{R}^d \times \mathbb{R}^d$ with marginals $\mu$ and $\nu$. Then, the Kantorovich-Rubinstein distance between $\mu$ and $\nu$ is given by
$$d_1(\mu, \nu) = \inf_{\pi \in M(\mu, \nu)} \int_{\mathbb{R}^d \times \mathbb{R}^d} || {\bf x} - {\bf y} ||_1 \, d \pi({\bf x}, {\bf y}).$$
\end{defn}

We point out that $d_1$ is only strictly speaking a distance when restricted to the subset of measures
$$\mathscr{P}_1(\mathbb{R}^d) := \left\{ \mu \in \mathscr{P}(\mathbb{R}^d): \int_{\mathbb{R}^d} || {\bf x} - {\bf x}_0 ||_1 \, d\mu( {\bf x}) < \infty \right\},$$
for some ${\bf x}_0 \in \mathbb{R}^d$, where $\mathscr{P}(\mathbb{R}^d)$ is the set of Borel probability measures on $\mathbb{R}^d$. We refer the interested reader to \cite{Villani_2009} for a thorough treatment of this distance, since Definition \ref{D.Wasserstein} gives only a special case.

An important property of the Kantorovich-Rubinstein distance is that if $\{ \mu_k \}_{k \in\mathbb{N}}$ is a sequence of probability measures in $\mathscr{P}_1(\mathbb{R}^d)$, then convergence in $d_1$ to a limit $\mu \in \mathscr{P}_1(\mathbb{R}^d)$ is equivalent to weak convergence. Furthermore, $d_1$ satisfies the useful {\bf duality formula}:
$$d_1(\mu, \nu) = \sup_{|| \psi ||_{\text{Lip}} \leq 1} \left\{ \int_{\mathbb{R}^d} \psi({\bf x}) d\mu({\bf x)} - \int_{\mathbb{R}^d} \psi({\bf x}) d\nu({\bf x)} \right\}$$
for all $\mu, \nu \in \mathscr{P}_1(\mathbb{R}^d)$, where the supremum is taken over al Lipschitz continuous functions $\psi: \mathbb{R}^d \to \mathbb{R}$ with Lipschitz constant one (see Remark 6.5 in \cite{Villani_2009}).

We now give the required assumption. With some abuse of notation, for joint distribution functions $F_n, F \in \mathbb{R}^d$ we write $d_1(F_n, F)$ to denote the Kantorovich-Rubinstein distance between their probability measures $\mu_n$ and $\mu$. The symbol $\stackrel{P}{\to}$ denotes convergence in probability.

\begin{assum} \label{A.WeakConvergence}
Given the extended bi-degree sequence $({\bf N}_n, {\bf D}_n, {\bf C}_n, {\bf Q}_n)$ define
$$F_n^*(m, q) :=  \frac{1}{n} \sum_{k=1}^n 1(N_k \leq m, Q_k \leq q) \quad \text{and} \quad F_n(m,q,x) := \sum_{k=1}^n 1(N_k \leq m, Q_k \leq q, C_k \leq x) \frac{D_k}{L_n} .$$
Suppose there exist random vectors $(\mathcal{N}_0, \mathcal{Q}_0)$ and $(\mathcal{N}, \mathcal{Q})$, and a random variable $\mathcal{C}$, such that
$$d_1(F_n^*, F^*) \stackrel{P}{\to} 0 \qquad \text{and} \qquad d_1(F_n, F) \stackrel{P}{\to} 0,$$
as $n \to \infty$, where
$$F^*(m,q) := P(\mathcal{N}_0 \leq m, \mathcal{Q}_0 \leq q) \qquad \text{and} \qquad F(m,q,x) :=  P(\mathcal{N} \leq m, \mathcal{Q} \leq q) P(\mathcal{C} \leq x).$$
\end{assum}

\begin{remark} \label{R.WeakConvergence}
Note that Assumption \ref{A.WeakConvergence} and the duality formula imply that
$$\sup \left\{ \mathbb{E}_n\left[ \psi(\hat N_1, \hat Q_1, \hat C_1) \right] - E[\psi( \mathcal{N}, \mathcal{Q}, \mathcal{C})] : \psi \, \text{\em is bounded and continuous} \right\}$$
converges to zero in probability, and therefore, by the bounded convergence theorem,
$$E\left[ \psi(\hat N_1, \hat Q_1, \hat C_1) \right] \to E[\psi( \mathcal{N}, \mathcal{Q}, \mathcal{C})], \quad n \to \infty,$$
for any  bounded and continuous function $\psi$, or equivalently, $(\hat N_1, \hat Q_1, \hat C_1) \Rightarrow (\mathcal{N}, \mathcal{Q}, \mathcal{C})$; similarly, $(\hat N_\emptyset, \hat Q_\emptyset) \Rightarrow (\mathcal{N}_0, \mathcal{Q}_0)$. The duality formula, combined with Assumption \ref{A.Bidegree}, also implies that $E[\mathcal{N}_0] = \nu_1$, $E[\mathcal{N}] = \mu$ and $E[\mathcal{C}] = \nu_5/\nu_1$.
\end{remark}

\subsection{Main Result} \label{S.MainResult}

We are now ready to state the main result of this paper, which establishes the convergence of the rank of a randomly chosen node in the DCM to a non-degenerate random variable $\mathcal{R}^*$.

\begin{theo} \label{T.Main}
Suppose the extended bi-degree sequence $({\bf N}_n, {\bf D}_n, {\bf C}_n, {\bf Q}_n)$ satisfies Assumptions \ref{A.Bidegree} and \ref{A.WeakConvergence}. Then,
$$R_1^{(n,\infty)} \Rightarrow \mathcal{R}^*$$
as $n \to \infty$, where $\mathcal{R}^*$ is defined as in \eqref{eq:LastLimit} with the weights $\{\mathcal{C}_i\}$ i.i.d. and independent of $(\mathcal{N}_0, \mathcal{Q}_0)$, respectively of $(\mathcal{N}, \mathcal{Q})$ in \eqref{eq:SFPE}.
\end{theo}

\begin{proof}
Define $\Omega_n$ according to Assumption \ref{A.Bidegree} and note that $P(\Omega_n^c) = O(n^{-\epsilon})$, so it
suffices to show that $R_1^{(n,\infty)}$, conditional on $\Omega_n$, converges weakly to $\mathcal{R}^*$.  Note that by Assumption \ref{A.Bidegree}, $\rho = E[\mathcal{N}] E[ |\mathcal{C}| ] = \nu_5 \mu / \nu_1 < 1$, which is a sufficient condition for $\mathcal{R}$ to be well defined (see Lemma~4.1 in \cite{Jel_Olv_12b}).  First, when $\mu > 1$, fix $0 < \delta < |\log c|/(2\log \mu)$ and let $k_n = s \log n$, where $\delta/|\log c| < s < 1/(2\log \mu)$. Next, note that by the arguments leading to \eqref{eq:TreeApprox},
\begin{align*}
P\left( \left.  \left| R_1^{(n,\infty)} - \hat R^{(n,k_n)}  \right| > n^{-\delta} \right| \Omega_n \right) &= O\left(n^\delta c^{k_n} + (\mu^{2k_n}/n)^{1/2} \right) \\
&= O\left( n^{\delta - s |\log c|} + n^{(2 s \log \mu -1)/2} \right) = o(1)
\end{align*}
as $n \to \infty$. When $\mu \leq 1$ we can take $k_n = n^\varepsilon$, with $\varepsilon < \min\{1/2, \gamma\}$, to obtain that the probability converges to zero. We then obtain that conditionally on $\Omega_n$,
$$\left| R_1^{(n,\infty)} - \hat R^{(n,k_n)}  \right| \Rightarrow 0.$$

That $\hat R^{(n,k_n)} \Rightarrow \mathcal{R}^*$ conditionally on $\Omega_n$ will follow from  Theorem 4.8 in \cite{Chen_Olv_14} and Assumption \ref{A.WeakConvergence} once we verify that, as $n \to \infty$,
\begin{equation} \label{eq:ConvergenceProducts}
\mathbb{E}_n \left[ \hat N_1 |\hat C_1| \right] \stackrel{P}{\to} E[ \mathcal{N}] E[ |\mathcal{C}|] \qquad \text{and} \qquad \mathbb{E}_n \left[ |\hat Q_1 \hat C_1 |\right] \stackrel{P}{\to} E[ |\mathcal{Q}|] E[ |\mathcal{C}|].
\end{equation}

To show that \eqref{eq:ConvergenceProducts} holds define $\phi_K(q,x) = (|q| \wedge K) (|x| \wedge 1)$ for $K > 0$, and note that since $\phi_K$ is bounded and continuous, Assumption \ref{A.WeakConvergence} and Remark \ref{R.WeakConvergence} imply that
$$\mathbb{E}_n\left[ \phi_K(\hat Q_1, \hat C_1) \right] \stackrel{P}{\to} E[ \phi_K(\mathcal{Q}, \mathcal{C}) ] = E[ |\mathcal{Q}| \wedge K] E[|\mathcal{C}|] , \qquad n \to \infty.$$
Next, fix $\epsilon > 0$ and choose $K$ such that $E[|\mathcal{Q}| 1(|\mathcal{Q}| > K)] < \epsilon/4$. Then,
\begin{align*}
 \left| \mathbb{E}_n\left[ |\hat Q_1 \hat C_1| \right] - E[|\mathcal{Q} \mathcal{C}|] \right|  &\leq  \left| \mathbb{E}_n\left[ \phi_K(\hat Q_1, \hat C_1) \right] - E[\phi_K(\mathcal{Q}, \mathcal{C})] \right| \\
 &\hspace{5mm} +  \mathbb{E}_n\left[ (|\hat Q_1| - K)^+  | \hat C_1| \right] +  E[ (|\mathcal{Q}| - K)^+  |\mathcal{C}|] \\
 &\leq \left| \mathbb{E}_n\left[ \phi_K(\hat Q_1, \hat C_1) \right] - E[\phi_K(\mathcal{Q}, \mathcal{C})] \right| +  c \mathbb{E}_n\left[ (|\hat Q_1|-K)^+ \right] + c \epsilon/4,
\end{align*}
where we used that both $|\hat C_1|$ and $|\mathcal{C}|$ are bounded by $c < 1$. It follows that
\begin{align*}
\lim_{n \to \infty} P\left( \left| \mathbb{E}_n\left[ |\hat Q_1 \hat C_1| \right] - E[|\mathcal{Q} \mathcal{C}|] \right| > \epsilon \right) &\leq \lim_{n \to \infty}  P\left(  \mathbb{E}_n\left[ (|\hat Q_1|-K)^+  \right]  > \epsilon/2  \right) .
\end{align*}
To show that this last limit is zero note that $(|x| - K)^+$ is Lipschitz continuous with Lipschitz constant one, so by the duality formula we obtain
$$\mathbb{E}_n\left[ (|\hat Q_1| - K)^+ \right] \stackrel{P}{\to} E[ (|\mathcal{Q}| - K)^+] < \epsilon/4$$
as $n \to \infty$, which gives the desired limit.

The proof for $\mathbb{E}_n \left[ |\hat N_1 \hat C_1 |\right]$ follows the same steps and is therefore omitted.
\end{proof}

\subsection{Asymptotic behavior of the limit}

We end this section by giving a limit theorem describing the tail asymptotics of $\mathcal{R}^*$; its proof is given in Section \ref{S.ProofAsymptotics}. This result covers the case where the weights $\left\{ \mathcal{C}_i\right\}$ are nonnegative and either the limiting in-degree $\mathcal{N}$ or the limiting personalization value $\mathcal{Q}$ have a regularly varying distribution, which in turn implies the regular variation of $\mathcal{R}$. Then, we deduce the asymptotics of $\mathcal{R}^*$ using some results for weighted random sums with heavy-tailed summands. The corresponding theorems can be found in \cite{Olvera_12a,Volk_Litv_10}.

\begin{defn}
We say that a  function $f$ is regularly varying at infinity with index $-\alpha$, denoted $f \in \mathscr{R}_{-\alpha}$, if $f(x) = x^{-\alpha} L(x)$ for some slowly varying function $L$; and $L:[0, \infty) \to (0, \infty)$ is slowly varying if $\lim_{x \to \infty} L(\lambda x)/L(x) = 1$ for any $\lambda > 0$.
\end{defn}

We use the notation $f(x) \sim g(x)$ as $x \to \infty$ for $\lim_{x \to \infty} f(x)/g(x) = 1$.

\begin{theo} \label{T.TailBehavior}
Suppose the generic branching vector $(\mathcal{N}, \mathcal{Q}, \mathcal{C}_1, \mathcal{C}_2, \dots)$ is such that the weights $\{ \mathcal{C}_i \}$ are nonnegative, bounded i.i.d. copies of $\mathcal{C}$, independent of $(\mathcal{N}, \mathcal{Q})$, $\mathcal{N} \in \mathbb{N}$ and $\mathcal{Q} \in \mathbb{R}$. Define $\rho = E[\mathcal{N}] E[\mathcal{C}]$ and $\rho_\alpha = E[\mathcal{N}] E[\mathcal{C}^\alpha]$ and let $\mathcal{R}$ be defined as in \eqref{eq:EndogenousSol}.
\begin{itemize}
\item If $P(\mathcal{N} > x) \in \mathscr{R}_{-\alpha}$, $\alpha > 1$, $\rho \vee \rho_\alpha < 1$, $P(\mathcal{N}_0 > x) \sim \kappa P(\mathcal{N} > x)$ as $x \to \infty$ for some $\kappa > 0$,  $E[ \mathcal{Q} ] , E[ \mathcal{Q}_0 ] > 0$, and $E\left[ |\mathcal{Q}|^{\alpha+\epsilon} + |\mathcal{Q}_0|^{\alpha+\epsilon} \right] < \infty$ for some $\epsilon > 0$, then
$$P(\mathcal{R}^* > x) \sim\left( E[\mathcal{N}_0] E[ \mathcal{C}^\alpha] +  \kappa  (1-\rho_\alpha)   \right)  \frac{(E[\mathcal{Q}]  E[ \mathcal{C}])^\alpha   }{(1-\rho)^\alpha (1-\rho_\alpha)} P(\mathcal{N} > x), \qquad x \to \infty.$$

\item If $P(\mathcal{Q} > x) \in \mathscr{R}_{-\alpha}$, $\alpha > 1$, $\rho \vee \rho_\alpha < 1$, $P(\mathcal{Q}_0 > x) \sim \kappa P(\mathcal{Q} > x)$ as $x \to \infty$ for some $\kappa > 0$,  $E[ |\mathcal{Q}|^\beta +|\mathcal{Q}_0|^\beta ] < \infty$ for all $0 < \beta < \alpha$, and $E\left[ |\mathcal{N}|^{\alpha+\epsilon} + |\mathcal{N}_0|^{\alpha+\epsilon} \right] < \infty$ for some $\epsilon > 0$, then
$$P(\mathcal{R}^* > x) \sim \left( E[\mathcal{N}_0] E[ \mathcal{C}^\alpha ] +  \kappa (1-\rho_\alpha) \right) (1-\rho_\alpha)^{-1} P(\mathcal{Q} > x) , \qquad x \to \infty.$$
\end{itemize}
\end{theo}

\begin{remark}
(i) For PageRank we have $C_i = c/D_i$ and $Q_i = 1-c$, where $c \in (0,1)$ is the damping factor. This leads to a limiting weight distribution of the form
$$P(\mathcal{C} \leq x) = \lim_{n \to \infty} \frac{1}{L_n} \sum_{i=1}^n 1(c/D_i \leq x) D_i,$$
which is not the limiting distribution of the reciprocal of the out-degrees, $\{ c/D_i\}$, but rather a size-biased version of it. 

(ii) Applying Theorem \ref{T.TailBehavior} to PageRank when $P(\mathcal{N} > x) \in \mathscr{R}_{-\alpha}$ and $P(\mathcal{N}_0 > x) \sim \kappa P(\mathcal{N} > x)$ for some constant $\kappa > 0$ gives that
$$P(\mathcal{R}^* > x) \sim \kappa' P(\mathcal{N} > x) \qquad \text{as } x \to \infty,$$
where $\kappa' > 0$ is determined by the theorem.

(iii) The theorem above only includes two possible cases of the relations between $(\mathcal{N}_0, \mathcal{Q}_0)$ and $(\mathcal{N}, \mathcal{Q})$. The exact asymptotics of $\mathcal{R}^*$ can be obtained from those of $\mathcal{R}$ in more cases than these using the same techniques; we leave the details to the reader.

(iv) Theorem \ref{T.TailBehavior} requires the weights $\left\{ \mathcal{C}_i\right\}$ to be nonnegative, which is not a condition in Theorem~\ref{T.Main}. The tail asymptotics of $\mathcal{R}$, and therefore of $\mathcal{R}^*$, in the real-valued case are unknown.
\end{remark}

\section{Algorithm to generate bi-degree sequences} \label{S.Example}

As an example of an extended bi-degree sequence satisfying Assumptions \ref{A.Bidegree} and \ref{A.WeakConvergence}, we give in this section an algorithm based on sequences of i.i.d. random variables.  The method for generating the bi-degree sequence $({\bf N}_n, {\bf D}_n)$ is taken from \cite{Chen_Olv_13}, where the goal was to generate a directed random graph with prescribed in- and out-degree distributions.

To define the algorithm we need to first specify target distributions for the in- and out-degrees, which we will denote by $f^{\text{in}}_k  = P(\mathscr{N} = k)$, and $f^\text{out}_k = P(\mathscr{D} = k)$, $k \geq 0$, respectively. Furthermore, we will assume that these target distributions satisfy $E[\mathscr{N}] = E[\mathscr{D}]$,
$$\overline{F^{\text{in}}}(x) = \sum_{k > x} f_k^{\text{in}} \leq x^{-\alpha} L_{\text{in}}(x) \qquad \text{and} \qquad \overline{F^{\text{out}}} (x) =  \sum_{k > x} f_k^{\text{out}} \leq x^{-\beta} L_{\text{out}}(x), $$
for some slowly varying functions $L_\text{in}$ and $L_\text{out}$, and $\alpha > 1, \beta > 2$.  To the original construction given in \cite{Chen_Olv_13} we will need to add two additional steps to generate the weight and personalization sequences ${\bf C}_n$ and ${\bf Q}_n$, for which we need two more distributions $F^\zeta(x) = P( \zeta \leq x)$ and $F^Q(x) = P( Q \leq x)$ with support on the real line and satisfying
$$P( |\zeta| \leq c ) = 1 \text{ for some $0 < c < 1$}, \quad \text{and} \quad E[|Q|^{1+ \epsilon_Q}] < \infty \text{ for some $0< \epsilon_Q \leq 1$}.$$

Let
$$\kappa_0 = \min\{ 1 - \alpha^{-1}, 1/2\}.$$

{\em The IID Algorithm:}
\begin{enumerate} \itemsep 0pt
\renewcommand{\labelenumi}{\arabic{enumi}.}
\item Fix $0 < \delta_0 < \kappa_0$.
\item Sample an i.i.d. sequence $\{\mathscr{N}_1, \dots, \mathscr{N}_n\}$ from distribution $F^\text{in}$; let $\overline{\mathscr N}_n=\sum_{i=1}^n \mathscr{N}_i$.

\item Sample an i.i.d. sequence $\{ \mathscr{D}_1, \dots, \mathscr{D}_n\}$ from distribution $F^\text{out}$, independent of $\{\mathscr{N}_i\}$; let $\overline{\mathscr D}_n=\sum_{i=1}^n \mathscr{D}_i$.

\item Define $\Delta_n= \overline{\mathscr N}_n- \overline{\mathscr D}_n$. If $|\Delta_n| \leq n^{1-\kappa_0 + \delta_0}$ proceed to step 5; otherwise repeat from step 2.

\item Choose randomly $|\Delta_n|$ nodes $\{i_1, i_2, \dots, i_{|\Delta_n|}\}$ without replacement and let
    \begin{align*}
        N_i &= \begin{cases}
        \mathscr{N}_i + 1 & \text{if $\Delta_n < 0$ and $i\in \{ i_1,i_2,\dots,i_{|\Delta_n|} \} $,}\\
        \mathscr{N}_i & \text{otherwise,}
    \end{cases}\\
    D_i &=\begin{cases}
        \mathscr{D}_i + 1 & \text{if $\Delta_n \ge 0$ and $i\in \{ i_1,i_2,\dots,i_{|\Delta_n|} \} $,}\\
        \mathscr{D}_i & \text{otherwise.}
    \end{cases}
  \end{align*}
  \item Sample an i.i.d. sequence $\{ Q_1, \dots, Q_n\}$ from distribution $F^Q$, independent of $\{\mathscr{N}_i\}$ and $\{\mathscr{D}_i\}$.
  \item Sample an i.i.d. sequence $\left\{ \zeta_1,\dots,\zeta_n \right\}$ from distribution $F^{\zeta}$, independent of $\{\mathscr{N}_i\}$, $\{\mathscr{D}_i\}$ and $\{Q_i\}$, and set $C_i=\zeta_i/D_i$ if $D_i \geq 1$ or $C_i = c \sgn(\zeta_i)$ otherwise.
\end{enumerate}

\begin{remark}
Note that since $E[|\mathscr{N} - \mathscr{D}|^{1+ a}] < \infty$ for any $0 < a < \min\{\alpha-1, \beta-1\}$, then $E[|\mathscr{N}-\mathscr{D}|^{1+ (\kappa_0-\delta_0)/(1-\kappa_0)}] < \infty$, and Corollary \ref{C.WeakLaw} in Section \ref{S.Proofs} gives
\begin{equation} \label{eq:Difference}
P\left( |\Delta_n| > n^{1-\kappa_0 + \delta_0} \right) = O\left(  n^{-\delta_0 (\kappa_0-\delta_0)/(1-\kappa_0) } \right)
\end{equation}
as $n \to \infty$.
\end{remark}

\bigskip

The two propositions below give the desired properties. Their proofs are given in Section \ref{S.ProofsIIDexample}.

\begin{prop} \label{P.VerifyAssumption1}
The extended bi-degree sequence $({\bf N}_n, {\bf D}_n, {\bf C}_n, {\bf Q}_n)$ generated by the IID Algorithm satisfies Assumption \ref{A.Bidegree} for any $0 < \kappa < \beta -2$, any $0 < \gamma < \min\{ (\kappa_0-\delta_0)^2/(1-\delta_0), \, (\beta-2-\kappa)/\beta\}$, $\mu = \nu_1 = E[\mathscr{N}] = E[\mathscr{D}]$, $\nu_2 = (E[\mathscr{D}])^2$, $\nu_3 = E[\mathscr{D}^2]$, $\nu_4 = E[ \mathscr{D}^{2+\kappa}]$, $\nu_5 = E[|\zeta|] P(\mathscr{D} \geq 1)$, $H = E[ |Q|] + 1$, and some $\varepsilon > 0$.
\end{prop}

\begin{prop} \label{P.VerifyAssumption2}
The extended bi-degree sequence $({\bf N}_n, {\bf D}_n, {\bf C}_n, {\bf Q}_n) $ generated by the IID Algorithm satisfies Assumption \ref{A.WeakConvergence} with
$$F^*(m,q) = P(\mathscr{N} \leq m) P(Q \leq q) \qquad \text{and}$$
$$F(m,q,x) = P(\mathscr{N} \leq m) P(Q \leq q) E[1( \zeta/\mathscr{D} \leq x) \mathscr{D}]/ \mu.$$
\end{prop}

\subsection{Numerical examples}

To complement the theoretical contribution of the paper, we use the IID Algorithm described in the previous section to provide some numerical results showing the accuracy of the WBP approximation to PageRank. To generate the in- and out-degrees we use the zeta distribution. More precisely, we set
\[
    \mathscr{N}_i = X_{1,i}+Y_{1,i},\quad \mathscr{D}_i = X_{2,i}+Y_{2,i},
\]
where $\{X_{1,i}\}$ and $\left\{ X_{2,i} \right\}$ are independent sequences of i.i.d. Zeta random variables with parameters $\alpha+1$ and $\beta+1$, respectively; $\{Y_{1,i}\}$ and $\left\{ Y_{2,i} \right\}$ are independent sequences of i.i.d.
Poisson random variables with different parameters chosen so that $\mathscr{N}$ and $\mathscr D$ have equal mean.
Note that the Poisson distribution has a light tail so that the power law tail behavior of $\mathscr{N}$ and $\mathscr{D}$ is preserved and determined by $\alpha$ and $\beta$, respectively.

Once the sequences $\left\{ \mathscr N_i \right\}$ and $\left\{ \mathscr D_i \right\}$ are generated, we use the IID Algorithm to obtain a valid bi-degree sequence $({\bf N}_n,{\bf D}_n)$.
Note that in PageRank, we have $\zeta_i=c$ and $Q_i=1-c$.
Given this bi-degree sequence we next proceed to construct the graph and the
TBT simultaneously, according to the rules described in Section \ref{S.CouplingWithTree}.
To compute $\mathbf R^{(n,\infty)}$ we perform matrix iterations with $r_0=1$ until
$\|\mathbf R^{(n,k)}-\mathbf R^{(n,k-1)}\|_2 <\varepsilon_0$
for some tolerance $\varepsilon_0$. We only generate the TBT for as many generations as it takes to construct the graph, with each generation corresponding to a step in the breadth first graph exploration process. The computation of the root node of the TBT, $\hat R^{(n,k)}$ is done recursively starting from the leaves using
\[
    \hat R_{\bf i}^{(n,0)}=1 \text{ for } {\bf i} \in \hat A_k, \qquad \hat R_{\bf i}^{(n,r)}=\sum_{j=1}^{\hat N_{\bf i}} \frac{c}{\hat D_{({\bf i},j)} } \hat R_{({\bf i},j)}^{(n,r-1)} + 1-c,  \text{ for } {\bf i} \in \hat A_r, \, 0 \leq r < k.
\]

To draw a sample from $\mathcal R^*$, note that by Proposition \ref{P.VerifyAssumption2}, $\mathcal{R}^*$ in the IID Algorithm has the same distribution as $\mathcal R$, i.e., the endogenous solution to the SFPE
$$\mathcal R\stackrel{\mathcal{D}}{=}\sum_{i=1}^{\mathscr N}\mathcal C_i \mathcal R_i+1-c,$$
where $P(\mathcal{C} \leq x) = E[1(c/\mathscr{D} \leq x) \mathscr{D}]/\mu$. To sample $\mathcal{R}$ we construct a WBP with generic branching vector $(\mathscr{N}, 1-c, \{ \mathcal{C}_i\})$, with the $\{\mathcal{C}_i\}$ i.i.d. and independent of $\mathscr{N}$ and proceed as in the computation of $\hat R^{(n,k)}$. To simulate samples of $\mathcal{C}$ we use the  acceptance-rejection method.

To show the convergence of $R_1^{(n,\infty)}$ to $\mathcal R^*$, we let $n=10$, 100 and 10000.
The values of the other parameters are $\alpha=1.5$, $\beta=2.5$, $\mathbb E[\mathscr N]=\mathbb E[\mathscr D]=2$, $c=0.3$. For the TBT, we simulate up to $k_n=\lfloor \log n\rfloor$ generations. For the WBP, we simulate 10 generations. For each $n$, we draw 1000 samples of $R_1^{(n,\infty)}$, $R_1^{(n,k_n)}$, $\hat R^{(n,k_n)}$ and $\mathcal R^*$, respectively, to approximate the distribution of these quantities.

Figure~\ref{fig:n=10000} shows the empirical CDFs of 1000 i.i.d. samples of the true PageRank, $R_1^{(n,\infty)}$; finitely many iterations of PageRank, $R_1^{(n,k_n)}$; and the TBT approximation $\hat R^{(n,k_n)}$; it also plots the distribution of  the limit $\mathcal R^*$ using $1000$ simulations. The approximations are so accurate that the CDFs are almost indistinguishable. Figure~\ref{fig:differentn} illustrates the weak convergence of PageRank on the graph, $R_1^{(n,\infty)}$, to its limit $\mathcal R^*$ as the size of the graph grows.

To quantify the distance between the CDFs, we sort the samples in ascending order and compute the mean squared error (MSE)
$\sum_{i=1}^{1000}(x_i^{(n)}-y_i)/1000$, where $y_i$ is the sorted $i$th sample of $\mathcal R^*$ and $x_i^{(n)}$ is the sorted $i$th sample of $R_1^{(n,\infty)}$. For robustness, we discard the squared error of the maximal value. As a result, the MSEs are 0.2950, 0.1813 and 0.0406 respectively for $n=10$, 100 and 10000. It is clear that the approximation improves as $n$ increases.

\begin{figure}[ht]
\centering
\includegraphics[scale=0.9, bb = 20 20 350 300, clip]{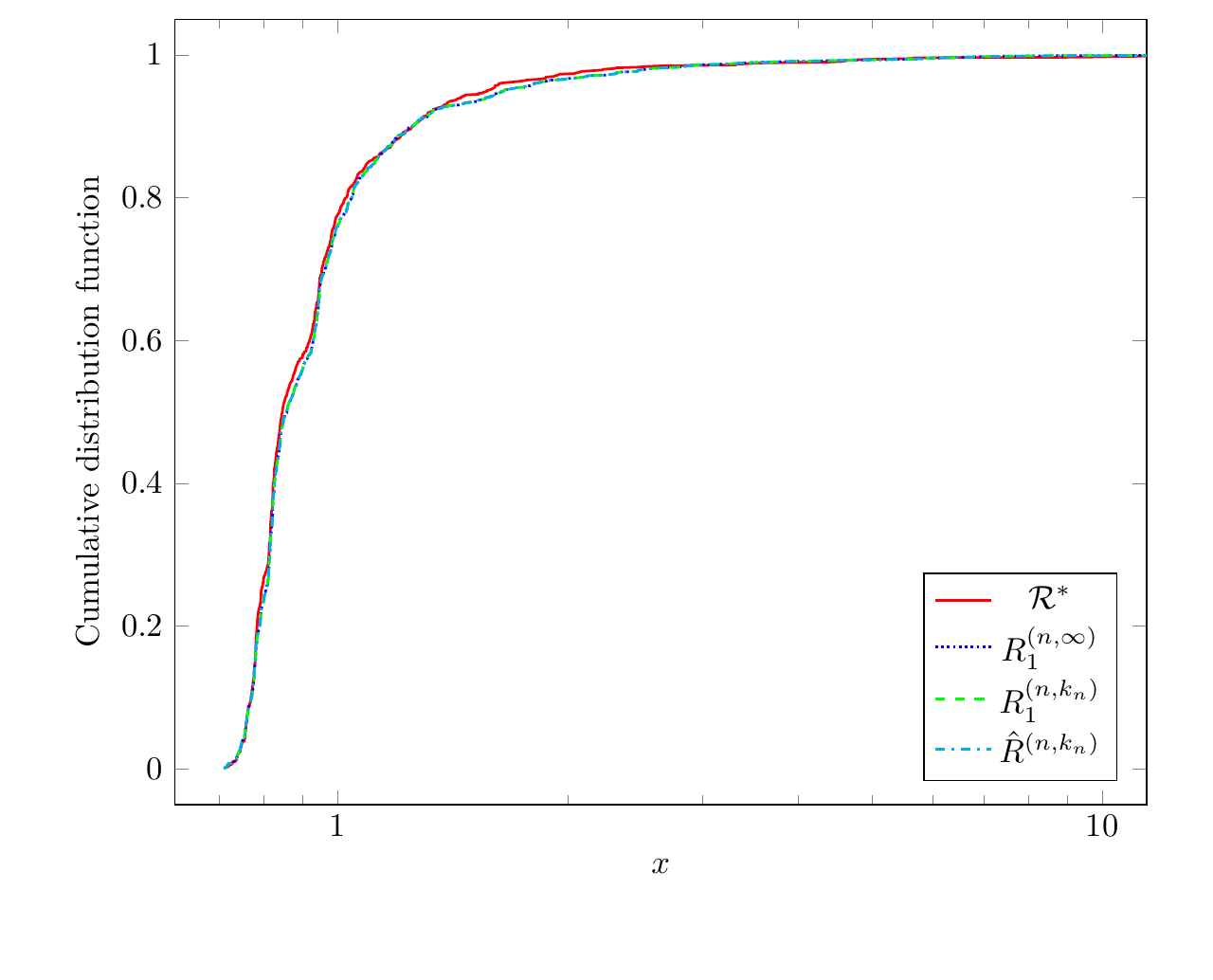}
\caption{The empirical CDFs of 1000 samples of $\mathcal R^*$, $R_1^{(n,\infty)}$, $R_1^{(n,k_n)}$ and $\hat R^{(n,k_n)}$ for $n=10000$ and $k_n=9$.} \label{fig:n=10000}
\end{figure}

\begin{figure}[ht]
\centering
\includegraphics[scale=0.9, bb = 20 20 350 300, clip]{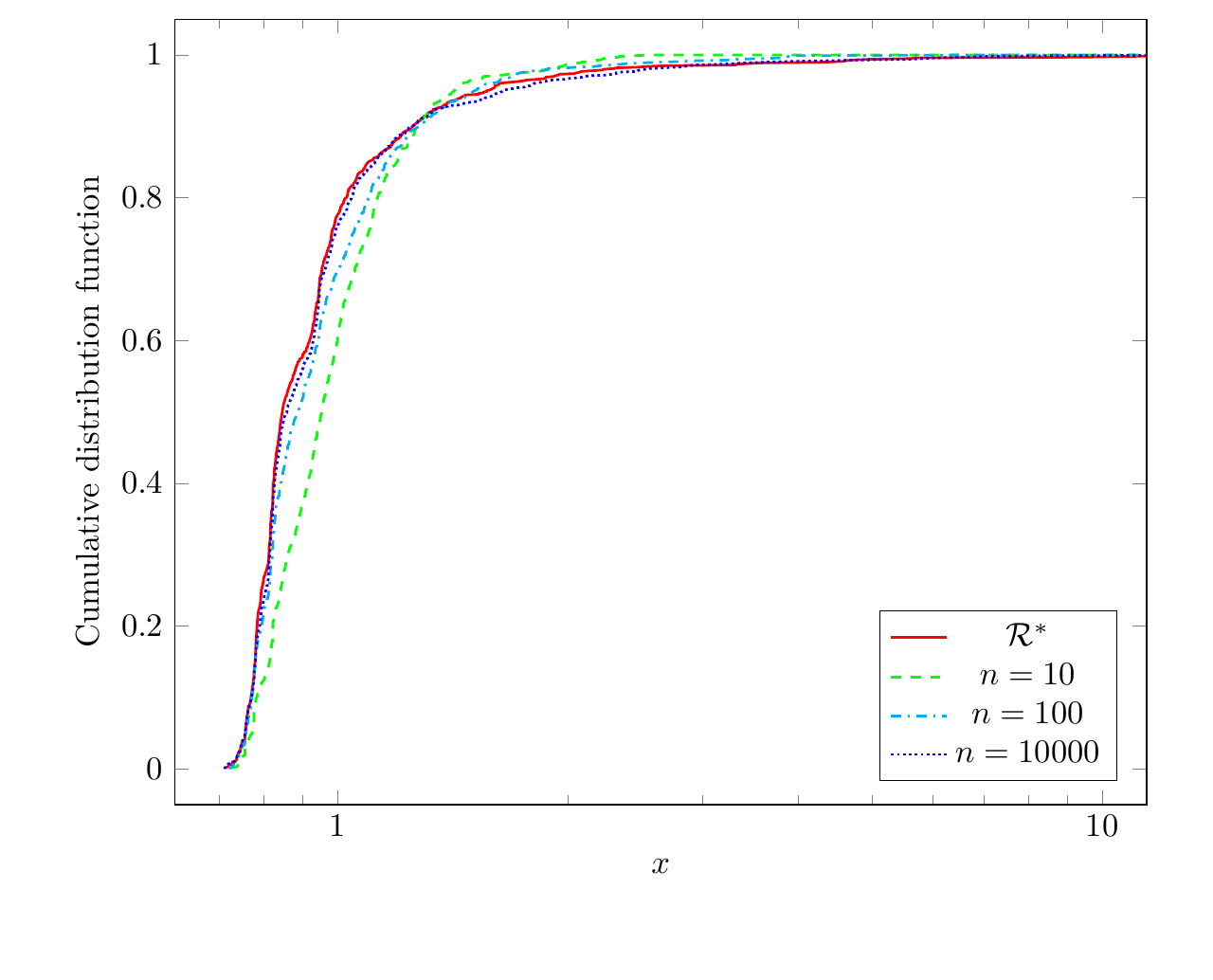}
\caption{The empirical CDFs of 1000 samples of $\mathcal R^*$ and $R_1^{(n,\infty)}$ for $n=10$, 100 and 10000.}
\label{fig:differentn}
\end{figure}

\section{Proofs} \label{S.Proofs}

The last section of the paper contains most of the proofs. For the reader's convenience we have
organized them in subsections according to the order in which their corresponding statements appear in the paper.

\subsection{Proof of the coupling lemma} \label{S.ProofCoupling}

Recall from Section \ref{S.CouplingWithTree} that $\hat N_\emptyset$ denotes the number of  offspring of the root node in the TBT (chosen from distribution \eqref{eq:FirstNodeDistr}) and $\hat N_1$ denotes the number of offspring of a node chosen from distribution \eqref{eq:randomJointDistr}. Throughout this section we will also need to define
\begin{align*}
\mu_n^* &= \mathbb{E}_n \left[ \hat N_\emptyset  \right] =  \sum_{i,j,s,t} i f_n^*(i,j,s,t) = \frac{1}{n} \sum_{k=1}^n N_k = \frac{L_n}{n},
\end{align*}
\vspace{-5pt}
and
\vspace{-5pt}
\begin{align*}
\mu_n &= \mathbb{E}_n\left[  \hat N_1  \right] = \sum_{i,j,s,t} i f_n(i,j,s,t) = \frac{1}{L_n} \sum_{k=1}^n N_k D_k.
\end{align*}

Before we give the proof of the Coupling Lemma \ref{L.CouplingBreaks} we will need the following estimates for the growth of the process $\{\hat Z_k\}$.

\begin{lemma} \label{L.Doob}
Suppose $({\bf N}_n, {\bf D}_n, {\bf C}_n, {\bf Q}_n)$ satisfies Assumption \ref{A.Bidegree} and recall that $\mu = \nu_2/\nu_1$. Then, for any constants $K > 0$, any nonnegative sequence $\{x_n\}$ with $x_n \to \infty$ and any $k = O( n^{\gamma})$,
$$P\left( \left. \max_{0 \leq r \leq k} \frac{\hat Z_r}{\mu^r} > K x_n  \right| \Omega_n \right) = O\left( x_n^{-1}  \right), \qquad n \to \infty.$$
\end{lemma}

\begin{proof}
Start by noting that for any $r = 0, 1, 2, \dots$,
\begin{equation} \label{eq:moment}
\mathbb{E}_n[ \hat Z_r ] = \mu_n^* \mu_n^{r}.
\end{equation}
Moreover, on the event $\Omega_n$,
\begin{align*}
\mu_n &= \frac{n \nu_2 (1 + O(n^{-\gamma}))}{n \nu_1 (1 + O(n^{-\gamma}))} = \mu (1 + O(n^{-\gamma})), \qquad \text{and} \\
\mu_n^* &= \frac{n \nu_1 (1 + O(n^{-\gamma}))}{n} = \nu_1 (1 + O(n^{-\gamma})).
\end{align*}

Next, note that conditionally on $\mathscr{F}_n$, the process
\begin{align*}
X_r &= \frac{\hat Z_r}{\mu_n^* \mu_n^r} = \frac{1}{\mu_n^* \mu_n^r} \sum_{{\bf i} \in \hat A_{r-1}} \hat N_{\bf i}, \quad r \geq 1, \qquad X_0 = \frac{\hat N_\emptyset}{\mu_n^*}
\end{align*}
is a nonnegative martingale with respect to the filtration $\sigma\left( \mathcal{F}_r \cup \mathscr{F}_n\right)$, where
$\mathcal{F}_r = \sigma\left( \hat N_{\bf i}: {\bf i} \in \hat A_s, \, s \leq r\right)$.  Therefore, we can apply Doob's inequality, conditionally on $\mathscr{F}_n$, to obtain
\begin{align*}
P\left( \left. \max_{0 \leq r \leq k} \frac{\hat Z_r}{\mu^r} > K x_n \right| \Omega_n \right) &= P\left( \left. \max_{0 \leq r \leq k} \frac{X_r \mu_n^* \mu_n^r}{\mu^r}  > K x_n  \right| \Omega_n \right) \\
&= P\left( \left. \max_{0 \leq r \leq k} X_r \nu_1 (1 + O(n^{-\gamma}))^{r+1}  > K x_n  \right| \Omega_n \right) \\
&\leq \frac{1}{P(\Omega_n)} E\left[ 1(\Omega_n) \mathbb{E}_n \left[ 1\left( \max_{0 \leq r \leq k} X_r > \frac{K x_n}{\nu_1 (1+O(n^{-\gamma}))^{k+1}} \right)   \right] \right] \\
&\leq \frac{1}{P(\Omega_n)} E\left[ 1(\Omega_n) \frac{\mathbb{E}_n[X_k] \nu_1 (1 + O(n^{-\gamma}))^{k+1}}{K x_n} \right] \\
&= \frac{\nu_1 (1 + O(n^{-\gamma}))^{k+1}}{K x_n} \qquad \text{(since $\mathbb{E}_n[ X_k] = 1$)}.
\end{align*}
Noting that $(1 + O( n^{-\gamma}))^{k} = e^{O(k n^{-\gamma})} = O(1)$ as $n \to \infty$ gives that this last term is $O(x_n^{-1})$. This completes the proof.
\end{proof}

\bigskip

We now give the proof of the coupling lemma.

\begin{proof}[Proof of Lemma \ref{L.CouplingBreaks}]
Start by defining
$$x_n = \begin{cases} (n/\mu^{2k})^{1/2}, & \mu > 1, \\
(n/k^2)^{1/2}, & \mu = 1, \\
n^{1/2}, & \mu < 1, \end{cases} \qquad \text{and} \qquad F_{k} = \left\{  \max_{0 \leq r \leq k} \frac{\hat Z_r}{\mu^r} \leq x_n  \right\}.$$
Note that $x_n \to \infty$ as $n \to \infty$ for all $1 \leq k \leq h \log n$ when $\mu > 1$ and for all $1\leq k \leq n^b$, $b < \min\{1/2, \gamma\}$, when $\mu \leq 1$. The constraint $b < \gamma$ will allow us to use Lemma \ref{L.Doob}.

Next, note that the $j$th inbound stub of node $i \in A_s$ (where the label $i$ refers to the order in which the node was added to the graph during the exploration process) will be the first one to be paired with an outbound stub having label 2 or 3 with probability
$$\frac{1}{L_n} \left(\sum_{r=0}^{s-1} \hat V_r + \sum_{t = 1}^{i-1} D_t + (j-1) \right) \leq \frac{1}{L_n} \sum_{r=0}^s \hat V_r =: P_s.$$
 It follows that,
 \begin{align*}
P(\tau = s | \Omega_n) &\leq P(\tau = s, F_{k} | \Omega_n) + P(\tau = s, F_{k}^c | \Omega_n) \\
&\leq  P( \text{Bin}( \hat Z_s, P_s) \geq 1, \, F_{k} | \Omega_n) + P(\tau = s, F_{k}^c | \Omega_n) ,
\end{align*}
where Bin$(n,p)$ is a Binomial random variable with parameters $(n,p)$. It follows that if we let $\mathcal{F}_k = \sigma( \hat Z_r, \hat V_r : 1 \leq r \leq k)$, then
\begin{align*}
P(\tau \leq k | \Omega_n) &= \sum_{s = 0}^{k} P(\tau = s | \Omega_n) \\
&\leq \sum_{s = 0}^{k} \left\{ P\left( \left. \text{Bin}( \hat Z_s, P_s) \geq 1, \,  F_{k} \right| \Omega_n \right) + P\left( \left.  \tau = s, \, F_{k}^c  \right| \Omega_n \right) \right\} \\
&\leq \sum_{s = 0}^{k}  E\left[ \left. 1(F_k) P( \text{Bin}(\hat Z_s, P_s) \geq 1 | \mathcal{F}_k)   \right| \Omega_n \right]   + P\left( \left.  F_{k}^c  \right| \Omega_n \right) \\
&\leq \sum_{s = 0}^{k}  E\left[ \left. 1(F_k) \hat Z_s P_s    \right| \Omega_n \right]   + P\left( \left.  F_{k}^c  \right| \Omega_n \right),
\end{align*}
where in the last step we used Markov's inequality. Now, use the bound for $\hat Z_s$ implied by $F_k$ and recall that $|\hat A_r| = \hat Z_{r-1}$ to obtain
\begin{align}
E\left[ \left. 1(F_k) \hat Z_s P_s    \right| \Omega_n \right] &\leq E\left[ \left. \mu^s x_n P_s    \right| \Omega_n \right] \label{eq:Zproduct}  \\
&= \frac{\mu^s x_n }{\nu_1 n} \sum_{r=0}^s E\left[ \left.  \hat V_r   \right| \Omega_n \right] (1+ O(n^{-\gamma})) \notag \\
&= \frac{\mu^s x_n}{\nu_1 n} \left\{  E\left[ \left. \hat V_0 \right| \Omega_n \right] + \sum_{r=1}^s E\left[ \left. \mathbb{E}_n\left[ \hat V_r | \hat Z_{r-1} \right]   \right| \Omega_n \right] \right\} (1+ O(n^{-\gamma})) \notag \\
&= \frac{\mu^s x_n}{\nu_1 n} \left\{  E\left[ \left. \mu_n^* \right| \Omega_n \right] + \sum_{r=1}^s E\left[ \left. \hat Z_{r-1} \lambda_n   \right| \Omega_n \right] \right\} (1+ O(n^{-\gamma})), \notag
\end{align}
where in the first equality we used that on the set $\Omega_n$ we have $L_n = \nu_1 n (1+ O(n^{-\gamma}))$, and on the second equality we used the observation that
$$\mathbb{E}_n \left[ \hat V_0 \right] = \mathbb{E}_n \left[ \hat D_\emptyset \right] = \mu_n^*, \qquad  \mathbb{E}_n \left[ \left.  \hat V_r \right| \hat Z_{r-1} \right] = \hat Z_{r-1} \lambda_n, \quad r \geq 1,$$
where $\lambda_n = \mathbb{E}_n [ \hat D_1]$. Moreover, on the set $\Omega_n$ we have that
$$\lambda_n = \frac{1}{L_n} \sum_{k=1}^n D_k^2 = \frac{n \nu_3 (1+ O(n^{-\gamma}))}{n \nu_1 (1 + O(n^{-\gamma}))} = \lambda (1 + O(n^{-\gamma})),$$
so we obtain
\begin{align}
E\left[ \left. 1(F_k) \hat Z_s P_s    \right| \Omega_n \right] &\leq \frac{\mu^s x_n}{\nu_1 n} \left\{  \nu_1 + \sum_{r=1}^s \lambda E\left[ \left. \hat Z_{r-1} \right| \Omega_n \right] \right\}  (1 + O(n^{-\gamma})) \notag \\
&= \frac{\mu^s x_n}{\nu_1 n} \left\{ \nu_1 +  \sum_{r=1}^s \lambda E\left[ \left. \mu_n^* \mu_n^{r-1} \right| \Omega_n\right] \right\}  (1 + O(n^{-\gamma}))  \qquad \text{(by \eqref{eq:moment})}. \notag
\end{align}
Using the observation that $E\left[ \left. \mu_n^* \mu_n^{r-1} \right| \Omega_n\right] = \nu_1 \mu^{r-1} (1 + O(n^{-\gamma}))^{r-1}$ (see the proof of Lemma \ref{L.Doob}), and the condition $r-1 < s \leq k = O(n^\gamma)$, gives
\begin{align*}
P(\tau \leq k | \Omega_n) &\leq (1 + O(1)) \frac{(\lambda+1) x_n}{ n} \sum_{s=0}^k \sum_{r=0}^s \mu^{s+r}   + P(F_k^c | \Omega_n).
\end{align*}
Note that we did not compute $E\left[ \left. \hat Z_s P_s \right| \Omega_n \right]$ in \eqref{eq:Zproduct} directly, since that would have led to having to compute $\mathbb{E}_n \left[ \hat Z_{s-1}^2 \right]$ and neither $\hat N_0$ nor $\hat N_1$ are required to have finite second moments in the limit. Now, since by Lemma \ref{L.Doob} we have that $P(F_k^c | \Omega_n) = O\left( x_n^{-1} \right)$, and
$$\sum_{s=0}^k  \sum_{r=0}^s \mu^{s+r} \leq \begin{cases} \mu^{2(k+1)} / (\mu-1)^2, & \mu > 1, \\
 (k+1)(k+2)/2, & \mu = 1, \\
 1/(1-\mu), & \mu < 1,
\end{cases}$$
we conclude that
$$P(\tau \leq k | \Omega_n) = \begin{cases} O\left( x_n \mu^{2k} n^{-1} + x_n^{-1} \right) = O\left( (n/\mu^{2k})^{-1/2} \right), & \mu > 1, \\
O\left( x_n k^2 n^{-1} + x_n^{-1} \right) = O\left( (n/k^2)^{-1/2} \right), & \mu = 1, \\
O\left( x_n n^{-1} + x_n^{-1} \right) = O\left( n^{-1/2} \right), & \mu < 1, \end{cases}$$
as $n \to \infty$. This completes the proof.
\end{proof}

\subsection{Proof of the asymptotic behavior of  $\mathcal{R}^*$} \label{S.ProofAsymptotics}

We give in this section the proof of Theorem \ref{T.TailBehavior} which describes the asymptotic behavior of the limit $\mathcal{R}^*$, which is essentially determined by the asymptotic behavior of the endogenous solution $\mathcal{R}$ given in \eqref{eq:EndogenousSol}. The tail behavior of $\mathcal{R}$ is the main focus of the work in \cite{Volk_Litv_10, Jel_Olv_10, Jel_Olv_12a, Jel_Olv_12b, Olvera_12a}.

\bigskip

\begin{proof}[Proof of Theorem \ref{T.TailBehavior}]
We consider the case when $\mathcal{N}$ is regularly varying first. By Theorem~3.4 in \cite{Olvera_12a} and the remarks that follow it (see also Theorem 4.1 in \cite{Volk_Litv_10}),
$$P(\mathcal{R} > x) \sim \frac{(E[\mathcal{Q}]  E[ \mathcal{C}_1])^\alpha   }{(1-\rho)^\alpha (1-\rho_\alpha)} P(\mathcal{N} > x), \qquad x \to \infty,$$
and therefore, $P(\mathcal{R} > x) \in \mathscr{R}_{-\alpha}$. Next,  since the $\{ \mathcal{C}_i\}$ are i.i.d. and independent of $\mathcal{N}$, Minkowski's inequality gives for any $\beta \geq 1$,
\begin{equation} \label{eq:Minkowski}
E \left[ \left( \sum_{i=1}^{\mathcal{N}} \mathcal{C}_i \right)^{\beta} \right] = E\left[ E\left[ \left. \left( \sum_{i=1}^{\mathcal{N}} \mathcal{C}_i \right)^{\beta} \right| \mathcal{N} \right] \right] \leq E\left[  \mathcal{N}^{\beta} E[ \mathcal{C}_1^{\beta}]  \right].
\end{equation}
Applying Lemma 2.3 in \cite{Olvera_12a} with $\beta = 1 + \delta$ gives that $E[|\mathcal{R}|^{1+\delta} ] < \infty$ for all $0 < \delta < \alpha-1$. By conditioning on the filtration $\mathcal{F}_k = \sigma \left( (\mathcal{N}_{\bf i},  \mathcal{C}_{({\bf i},1)}, \mathcal{C}_{({\bf i},2)}, \dots):  {\bf i} \in \mathcal{A}_s, s < k \right)$ it can be shown that $E\left[ \sum_{{\bf i} \in \mathcal{A}_k} \Pi_{\bf i} \mathcal{Q}_{\bf i} \right] = \rho^k E[\mathcal{Q}]$, which implies that $E[ \mathcal{R} ] = (1-\rho)^{-1} E[\mathcal{Q}] > 0$. Also, by Lemma 3.7(2) in \cite{Jess_Miko_06} we have
$$P\left( \sum_{i=1}^{\mathcal{N}_0} \mathcal{C}_i > x\right) \sim \left( E[ \mathcal{C}_1 ] \right)^\alpha P(\mathcal{N}_0 > x) \sim  \kappa \frac{(1-\rho)^\alpha (1-\rho_\alpha)}{(E[\mathcal{Q}] )^\alpha } P(\mathcal{R} > x).$$
Using Theorem A.1 in \cite{Olvera_12a} we conclude that
\begin{align*}
P(\mathcal{R}^* > x) &\sim \left( E[\mathcal{N}_0] E[ \mathcal{C}_1^\alpha] + \kappa \frac{(1-\rho)^\alpha (1-\rho_\alpha)}{(E[\mathcal{Q}])^\alpha }  \left( E[\mathcal{R}] \right)^\alpha \right) P(\mathcal{R} > x) \\
&\sim \left( E[\mathcal{N}_0] E[ \mathcal{C}_1^\alpha] +  \kappa  (1-\rho_\alpha)   \right)  \frac{(E[\mathcal{Q}]  E[ \mathcal{C}_1])^\alpha   }{(1-\rho)^\alpha (1-\rho_\alpha)} P(\mathcal{N} > x)
\end{align*}
as $x \to \infty$.

Now, for the case when $\mathcal{Q}$ is regularly varying, note that $E \left[ \left( \sum_{i=1}^{\mathcal{N}} \mathcal{C}_i \right)^{\alpha+\epsilon} \right]< \infty$ by \eqref{eq:Minkowski} and the theorem's assumptions. Then, by Theorem 4.4 in \cite{Olvera_12a} (see also Theorem 4.1 in \cite{Volk_Litv_10}) we have
$$P(\mathcal{R} > x) \sim (1-\rho_\alpha)^{-1} P( \mathcal{Q} > x) , \qquad x \to \infty.$$
The same observations made for the previous case give $E[ |\mathcal{R}|^{1+\delta}] < \infty$ for all $0 < \delta < \alpha-1$. In addition, note that the same argument used above gives $E \left[ \left( \sum_{i=1}^{\mathcal{N}_0} \mathcal{C}_i \right)^{\alpha+\epsilon} \right] < \infty$. Also,
$$P\left( \mathcal{Q}_0 > x \right) \sim \kappa P\left( \mathcal{Q} > x \right) \sim \kappa (1-\rho_\alpha) P(\mathcal{R} > x).$$
It follows, by Theorem A.2 in \cite{Olvera_12a}, that
\begin{align*}
P\left( \mathcal{R}^* > x \right) &\sim \left( E[\mathcal{N}_0] E[ \mathcal{C}_1^\alpha ] +  \kappa (1-\rho_\alpha)  \right) P(\mathcal{R} > x) \\
&\sim \left( E[\mathcal{N}_0] E[ \mathcal{C}_1^\alpha ] +  \kappa (1-\rho_\alpha) \right) (1-\rho_\alpha)^{-1} P(\mathcal{Q} > x)
\end{align*}
as $x \to \infty$.
\end{proof}

\subsection{Proofs of properties of the {\em IID Algorithm}} \label{S.ProofsIIDexample}

Before giving the proofs of Propositions \ref{P.VerifyAssumption1} and \ref{P.VerifyAssumption2} we will need some general results for sequences of i.i.d. random variables, which may be of independent interest. The first result establishes a bound for the sum of the largest order statistics in a sample. The second result is essentially an explicit version of the Weak Law of Large Numbers.

\begin{lemma} \label{lem:order-stat-converge}
Let $X_1,X_2,\dots,X_n$ be i.i.d. nonnegative random variables satisfying $E[X_1^{1+\kappa}] < \infty$ for some $\kappa > 0$, and let $X_{(i)}$ denote the $i$th smallest observation from the set $\{X_1, X_2, \dots, X_n\}$.  Let $\left\{ \pi_1,\pi_2,\dots,\pi_n \right\}$ be any permutation of the set $\left\{ 1,2,\dots,n \right\}$. Then, for any $k_n \in \{1,2, 3, 4, \dots, n\}$ we have
$$P\left( \sum_{i= n-k_n+1}^n X_{(i)} > n^{1-\gamma} \right) = O\left( k_n^{\kappa/(1+\kappa)} n^{-(\kappa/(1+\kappa) -\gamma)} \right)$$
as $n \to \infty$.
\end{lemma}

\begin{proof}
Note that, by Markov's inequality,
$$P(X_1 > x) \leq E[X_1^{1+\kappa}] x^{-1-\kappa},$$
and therefore,
$$P(X_i > x) \leq P(Y_i > x),$$
where $\{ Y_1, Y_2, \dots, Y_n\}$ are i.i.d. Pareto random variables having distribution $G(x) = 1 - (x/b)^{-1-\kappa}$ for $x > b := \left(E[X_1^{1+\kappa}] \right)^{-1/(1+\kappa)}$. We then have that
\begin{align*}
P\left( \sum_{i= n-k_n+1}^{n} X_{(i)} > n^{1-\gamma} \right)  &\leq P\left( \sum_{i= n-k_n+1}^{n} Y_{(i)} > n^{1-\gamma} \right) \\
&\leq \frac{1}{n^{1-\gamma}}  \sum_{i=n-k_n+1}^n E[ Y_{(i)} ],
\end{align*}
where $Y_{(i)}$ is the $i$th smallest from the set $\{Y_1, Y_2, \dots, Y_n\}$. Moreover, it is known (see \cite{Vannman_76}, for example) that
$$E[ Y_{(i)} ] = b \cdot \frac{n!}{(n-i)!} \cdot \frac{\Gamma(n-i+1 - (1+\kappa)^{-1})}{\Gamma(n+1-(1+\kappa)^{-1})},$$
where $\Gamma(\cdot)$ is the Gamma function. By Wendel's inequality \cite{Wendel_48}, for any $0 < s < 1$ and $x > 0$,
$$\left( \frac{x}{x+s} \right)^{1-s} \leq \frac{\Gamma(x+s)}{x^s \Gamma(x)} \leq 1,$$
and therefore, for $i < n$, and $\vartheta = (1+\kappa)^{-1}$,
$$E[ Y_{(i)} ]  \leq b \cdot \frac{n!}{ \Gamma(n+1- \vartheta)} \cdot \frac{1}{(n-i)^{\vartheta}} \leq b\left( \frac{n+ 1 - \vartheta}{n-i} \right)^{\vartheta}.$$
We conclude that
\begin{align*}
\frac{1}{n^{1-\gamma}}  \sum_{i=n-k_n+1}^n E[ Y_{(i)} ] &\leq \frac{b}{n^{1-\gamma}} \left( \sum_{i=n-k_n+1}^{n-1}  \left( \frac{n+1-\vartheta}{n-i} \right)^{\vartheta} + \frac{n!  \Gamma(1-\vartheta) }{\Gamma(n+1-\vartheta)} \right) \\
&\leq \frac{b (n+1- \vartheta)^{\vartheta} }{n^{1-\gamma}} \left( \sum_{i=n-k_n+1}^{n-1}  \left( \frac{1}{n-i} \right)^{\vartheta} + \Gamma(1-\vartheta) \right) \\
&\leq \frac{b (n+1)^{\vartheta} }{n^{1-\gamma}} \left( \sum_{j=1}^{k_n-1} \int_{j-1}^j   \frac{1}{t^{\vartheta}} \, dt + \Gamma(1-\vartheta) \right) \\
&= \frac{b (n+1)^\vartheta }{n^{1-\gamma}} \left( \frac{(k_n-1)^{1-\vartheta}}{1-\vartheta} + \Gamma(1-\vartheta) \right) \\
&= O\left( \frac{k_n^{1-\vartheta}}{n^{1-\vartheta-\gamma}} \right),
\end{align*}
where in the second inequality we used Wendel's inequality. This completes the proof.
\end{proof}

\begin{lemma} \label{L.WeakLaw}
Let $\{X_1, X_2, \dots, X_n\}$ be i.i.d. random variables satisfying $E[|X_1|^{1+\kappa}] < \infty$ for some $\kappa > 0$ and $\mu = E[X_1]$.  Set $S_m = X_1 + \dots + X_m$ and  $\theta = \min\{ 1+\kappa, 2 \}$. Then, for any $K > 0$, any nonnegative sequence $\{x_n\}$ such that $x_n \to \infty$ as $n \to \infty$, and all $m = o\left( x_n^{1+\kappa} \right)$, there exists an $n_0 \geq 1$ such that for all $n \geq n_0$ ,
$$P\left(   |S_m - m \mu | > K x_n \right) \leq E[|X_1|^\theta] \left( \frac{2}{K^2} + 1 \right) \frac{m}{x_n^\theta}.$$
\end{lemma}

\begin{proof}
If $\kappa \geq 1$, then Chebyshev's inequality gives, for all $m \geq 1$,
$$P\left( |S_m - m\mu| > K x_n \right) \leq \frac{m\var(X_1)}{K^2 x_n^2} \leq \frac{m E[|X_1|^2]}{K^2 x_n^2} = \frac{m E[|X_1|^\theta]}{K^2 x_n^\theta}.$$

Suppose now that $0 < \kappa < 1$ and let $G(t) = P(|X_1| \leq t)$. Set $t = x_n$ and define $P(\tilde X_i \leq x) = P(X_i \leq x | X_i \leq t)$, and note that
\begin{align*}
\left| E[\tilde X_1] - \mu \right| &= \left| E[ X_1 1(|X_1| \leq t)] / G(t) - \mu \right| \\
&\leq \frac{1}{G(t)} \left|  E[ X_1 1(|X_1| \leq t)] - \mu  \right|  + \frac{|\mu| \overline{G}(t)}{G(t)} \\
&= \frac{1}{G(t)} \left( \left| E[X_1 1(|X_1| > t)] \right| + |\mu| \overline{G}(t) \right) \\
&\leq \frac{1}{G(t)} \left(  t \overline{G}(t) +  \int_t^\infty \overline{G}(x)  dx   + |\mu| \overline{G}(t) \right) \\
&\leq \frac{E[|X_1|^{1+\kappa}]}{G(t)} \left( t^{-\kappa} + \int_t^\infty x^{-1-\kappa} \, dx + |\mu| t^{-1-\kappa} \right) \qquad \text{(by Markov's inequality)} \\
&= \frac{E[|X_1|^{1+\kappa}]}{G(t)} \left( \frac{1+\kappa}{\kappa}  + |\mu| t^{-1} \right) t^{-\kappa}.
\end{align*}
Then, for sufficiently large $n$, we obtain that
$$\left| E[\tilde X_1] - \mu \right| \leq 2 E[|X_1|^{1+\kappa}] \left( \frac{1+\kappa}{\kappa} + |\mu| \right) t^{-\kappa} \triangleq K' t^{-\kappa} = K' x_n^{-\kappa}.$$
It follows that for sufficiently large $n$ and $m = o(x_n^{1+\kappa})$,
\begin{align*}
&P\left(  \left| S_m - m\mu \right| > K x_n \right) \\
&= P\left( \left| \sum_{i=1}^m ( \tilde X_i - \mu) \right|  > K x_n \right) G(t)^m  + P\left( \left| \sum_{i=1}^m (X_i - \mu) \right| > K x_n, \, \max_{1\leq i \leq m} |X_i| > t  \right) \\
&\leq P\left( \left| \sum_{i=1}^m ( \tilde X_i - E[\tilde X_1] ) \right|  +m \left| E[\tilde X_1] - \mu \right| > K x_n \right) G(t)^m  + P\left( \max_{1\leq i \leq m} |X_i| > t  \right) \\
&\leq \frac{G(t)^m}{\left( K x_n - K' m t^{-\kappa} \right)^2} \cdot m \var(\tilde X_1)+ 1 - G(t)^m \qquad \text{(by Chebyshev's inequality)}\\
&\leq \frac{G(t)^m m \var(\tilde X_1)}{K^2 x_n^2 (1 - m x_n^{-1-\kappa} K'/K)^2} + m \overline{G}(t).
\end{align*}
To estimate $\var(\tilde X_1)$ note that
$$\var(\tilde X_1) \leq E[ \tilde X_1^2 ] = \frac{E[ X_1^2 1(|X_1|\leq t)]}{G(t)} \leq \frac{E[|X_1|^{1+\kappa}] t^{1-\kappa}}{G(t)},$$
so using Markov's inequality again to estimate $\overline{G}(t)$ gives us
\begin{align*}
P\left(  \left| S_m - m\mu \right| > K x_n \right) &\leq \frac{E[|X_1|^{1+\kappa}] }{K^2 (1-  m x_n^{-1-\kappa} K'/K)^2} \cdot \frac{m t^{1-\kappa}}{x_n^2} + \frac{E[|X_1|^{1+\kappa}] m}{t^{1+\kappa}} \\
&=   E[|X_1|^{1+\kappa}] \left( \frac{1}{K^2 (1- m x_n^{-1-\kappa} K'/K)^2} + 1\right) \frac{m}{x_n^{1+\kappa}} \\
&= E[|X_1|^\theta ] \left( \frac{1}{K^2 (1- m x_n^{-1-\kappa} K'/K)^2} + 1\right) \frac{m}{x_n^\theta}  .
\end{align*}
This completes the proof.
\end{proof}

\bigskip

By setting $m = n$ and $x_n = n^{1-\gamma}$ we immediately obtain the following corollary.

\begin{cor} \label{C.WeakLaw}
Let $\{X_1, X_2, \dots, X_n\}$ be i.i.d. random variables satisfying $E[|X_1|^{1+\kappa}] < \infty$ for some $\kappa > 0$ and $\mu = E[X_1]$.  Set $S_n = X_1 + \dots + X_n$. Then, for any $0 \leq \gamma < 1- 1/\theta$ , $\theta = \min\{ 1+\kappa, 2 \}$ and any constant $K > 0$, there exists an $n_0 \geq 1$ such that for all $n \geq n_0$
$$P\left(   |S_n - n \mu | > K n^{1-\gamma} \right) \leq E[|X_1|^\theta] \left( \frac{2}{K^2} + 1 \right) n^{-\theta(1-1/\theta-\gamma)}.$$
\end{cor}

We now proceed to prove that the extended bi-degree sequence generated by the IID Algorithm satisfies Assumptions \ref{A.Bidegree} and \ref{A.WeakConvergence}.

\bigskip

\begin{proof}[Proof of Proposition \ref{P.VerifyAssumption1}]
It suffices to show that $P\left(\Omega_{n,i}^c\right)=O(n^{-\varepsilon})$ for  some $\varepsilon > 0$ and $i=1,\dots,6$.  Throughout the proof let $E_n = \{ |\Delta_n| \leq n^{1-\kappa_0+\delta_0} \}$ and recall that by \eqref{eq:Difference} $P(E_n^c ) = O\left( n^{-\delta_0 \eta} \right)$, where $\eta = (\kappa_0-\delta_0)/(1-\kappa_0)$.

We start with $\Omega_{n,2}$. Let $\nu_2 = (E[\mathscr{D}])^2$ and define $\chi_i = D_i - \mathscr{D}_i$, $\tau_i = N_i - \mathscr{N}_i$. Note that $\chi_i, \tau_i \in \{0, 1\}$ for all $i = 1, \dots, n$; moreover, either all the $\{\chi_i\}$ or all the $\{\tau_i\}$ are zero, and therefore $\chi_i \tau_j = 0$ for all $1\leq i,j \leq n$. We now have
\begin{align*}
        \left|\sum_{i=1}^n D_{i} N_{i} - n \nu_2 \right| &= \left|\sum_{i=1}^n \mathscr{D}_{i} \mathscr{N}_i - n \nu_2 +\sum_{i=1}^{n} (\mathscr{D}_i \tau_i + \chi_i \mathscr{N}_i  )  \right| \\
        &\leq  \left|\sum_{i=1}^n \mathscr{D}_{i} \mathscr{N}_i - n \nu_2 \right|+ \max\left\{ \sum_{i=n-\Delta_n+1}^{n} \mathscr{D}_{(i)}, \, \sum_{i = n - \Delta_n+1}^n \mathscr{N}_{(i)}  \right\},
    \end{align*}
  where $\mathscr{D}_{(i)}$ (respectively, $\mathscr{N}_{(i)}$) is the $i$th smallest value from the set $\{ \mathscr{D}_1, \dots, \mathscr{D}_n\}$ (respectively, $\{ \mathscr{N}_1, \dots, \mathscr{N}_n\}$). Since $|\Delta_n| \leq n^{1-\kappa_0+\delta_0}$ on $E_n$, we have
 \begin{align*}
 P(\Omega_{n,2}^c) &=  P\left( \left. \left|\sum_{i=1}^n D_{i} N_i - n \nu_2 \right| > n^{1-\gamma} \right| E_n \right) \\
 &\leq \frac{1}{P(E_n)} \left\{  P\left( \left|\sum_{i=1}^n \mathscr{D}_{i} \mathscr{N}_i - n \nu_2 \right| > \frac{n^{1-\gamma}}{2} \right) \right. \\
 &\hspace{5mm} \left. + P\left(  \sum_{i=n- \lfloor n^{1-\eta(1-\kappa_0)} \rfloor +1 }^{n} \mathscr{D}_{(i)} > \frac{n^{1-\gamma}}{2} \right) + P\left(  \sum_{i=n- \lfloor n^{1-\eta(1-\kappa_0)} \rfloor+1}^{n} \mathscr{N}_{(i)} > \frac{n^{1-\gamma}}{2} \right) \right\}.
 \end{align*}
 Now apply Corollary \ref{C.WeakLaw} to $X_i = \mathscr{D}_i \mathscr{N}_i$, which satisfies $E[ (\mathscr{D}_1 \mathscr{N}_1)^{1+\eta}] = E[ \mathscr{N}_1^{1+\eta}]  E[ \mathscr{D}_1^{1+\eta}]  < \infty$, to obtain
$$P\left( \left|\sum_{i=1}^n \mathscr{D}_{i} \mathscr{N}_i - n \nu_2 \right| > \frac{n^{1-\gamma}}{2} \right) = O\left( n^{-\eta  + (1+\eta) \gamma} \right).$$
For the remaining two probabilities use Lemma \ref{lem:order-stat-converge} to see that
\begin{align*}
&P\left(  \sum_{i=n- \lfloor n^{1-\eta(1-\kappa_0)} \rfloor +1 }^{n} \mathscr{D}_{(i)} > \frac{n^{1-\gamma}}{2} \right) + P\left(  \sum_{i=n- \lfloor n^{1-\eta(1-\kappa_0)} \rfloor+1}^{n} \mathscr{N}_{(i)} > \frac{n^{1-\gamma}}{2} \right) \\
&= O\left( n^{(1-\eta(1-\kappa_0)) \eta/(1+\eta) - (\eta/(1+\eta) -\gamma)} \right) \\
&= O \left( n^{-\eta(\kappa_0-\delta_0)/(1+\eta) + \gamma} \right).
\end{align*}
It follows from these estimates that
\begin{equation} \label{eq:Omega2}
P(\Omega_{n,2}^c) = O\left( n^{-\eta(\kappa_0-\delta_0)/(1+\eta) + \gamma} \right) .
\end{equation}

Next, we can analyze $\Omega_{n,1}, \Omega_{n,3}$ and $\Omega_{n,4}$ by considering the sequence $\{D_{i}^{\vartheta} \}$ where $\vartheta$ can be taken to be $1, 2$ or $2+\kappa$. Correspondingly, we have $\nu_1 = E[ \mathscr{D}]$, $\nu_3 = E[\mathscr{D}^2]$ and $\nu_4 = E[\mathscr{D}^{2+\kappa}]$. Similarly as what was done for $\Omega_{n,2}$, note that
\begin{align*}
\left| \sum_{i=1}^n D_i^{\vartheta} - n E[\mathscr{D}^{\vartheta}] \right| &\leq \left| \sum_{i=1}^n \mathscr{D}_i^{\vartheta} - n E[\mathscr{D}^{\vartheta}] \right| +  \sum_{i=1}^n \left( (\mathscr{D}_i + \chi_i)^{\vartheta} - \mathscr{D}_i^\vartheta \right) \\
&\leq \left| \sum_{i=1}^n \mathscr{D}_i^{\vartheta} - n E[\mathscr{D}^{\vartheta}] \right| +  \sum_{i=1}^n \vartheta (\mathscr{D}_i + 1)^{\vartheta-1} \chi_i,
\end{align*}
where we used the inequality $(d + x)^\vartheta - d^\vartheta \leq \vartheta (d+1)^{\vartheta-1} x$ for $d \geq 0$, $x \in [0,1]$ and $\vartheta \geq 1$. Now note that $E[(\mathscr{D}^\vartheta)^{1+\sigma}] < \infty$ for any $0<\sigma < (\beta-2-\kappa)/(2+\kappa)$; in particular, since $\gamma < (\beta-2-\kappa)/\beta$, we can choose $\gamma/(1-\gamma) < \sigma < (\beta-2-\kappa)/(2+\kappa)$. For such $\sigma$, Corollary \ref{C.WeakLaw} gives
$$P\left( \left| \sum_{i=1}^n \mathscr{D}_i^{\vartheta} - n E[\mathscr{D}^{\vartheta}] \right| > \frac{n^{1-\gamma}}{2} \right) = O \left( n^{-\sigma + (1+\sigma) \gamma}   \right).$$
For the term involving the $\{\chi_i\}$ we use again Lemma \ref{lem:order-stat-converge} to obtain
\begin{align*}
P\left(  \sum_{i=1}^n \vartheta (\mathscr{D}_i + 1)^{\vartheta-1} \chi_i > \frac{n^{1-\gamma}}{2}  \right) &\leq P\left(  \sum_{i=n- \lceil n^{1-\eta} \rceil+1}^n \vartheta (\mathscr{D}_{(i)} + 1)^{\vartheta-1} > \frac{n^{1-\gamma}}{2}  \right) \\
&= O\left( n^{(1-\eta)(1 - 1/2) - (1 - \gamma -1/2)}   \right) \\
&= O\left(  n^{ - \eta/2 + \gamma}   \right).
\end{align*}
It follows that
\begin{align}
P(\Omega_{n,i}^c) &\leq \frac{1}{P(E_n)} \cdot  O\left( n^{-\sigma+(1+\sigma)\gamma} + n^{-\eta/2+\gamma} \right)  , \qquad i = 1, 3, 4. \label{eq:Omega134}
\end{align}

Now note that since $|\zeta| \leq c < 1$ a.s., then $E[|\zeta|^2] < \infty$ and Corollary \ref{C.WeakLaw} gives
\begin{align} \label{eq:Omega5}
P(\Omega_{n,5}^c) &= P\left( \left| \sum_{r=1}^n |\zeta_r| 1(D_r \geq 1)   - n\nu_5 \right| > n^{1-\gamma}  \right) \notag \\
&= P\left( \left| \sum_{r=1}^n |\zeta_r| 1(\mathscr{D}_r \geq 1)  - n\nu_5 \right| + c|\Delta_n|  > n^{1-\gamma}  \right) =  O\left( n^{-1 + 2\gamma} \right).
\end{align}

Finally, by Corollary \ref{C.WeakLaw} and \eqref{eq:Difference},
\begin{equation} \label{eq:Omega6}
P(\Omega_{n,6}^c) \leq P\left( \left. \left| \sum_{r=1}^n |Q_r| - nE[|Q|] \right| > n \right| E_n \right) = O\left(  n^{-\epsilon_Q}  + n^{-\delta_0\eta} \right).
\end{equation}

Our choice of $0 < \gamma  < \min\{ \eta(\kappa_0-\delta_0)(1+\eta), \, \sigma/(1+\sigma)\}$ guarantees that all the exponents of $n$ in expressions \eqref{eq:Omega2} - \eqref{eq:Omega5} are strictly negative, which completes the proof.
 \end{proof}

 \bigskip

\begin{proof}[Proof of Proposition \ref{P.VerifyAssumption2}]
We will show that $d_1(F_n^*, F^*)$ and $d_1(F_n, F)$ converge to zero a.s. by using the duality formula for the Kantorovich-Rubinstein distance.  To this end, let $S_n = \sum_{i=1}^n \mathscr{D}_i$, $\mathscr{C}_k = \zeta_k/\mathscr{D}_k 1(\mathscr{D}_k \geq 1) + c \sgn(\zeta_k) 1(\mathscr{D}_k = 0)$,  and fix $\psi^*: \mathbb{R}^2 \to \mathbb{R}$ and $\psi: \mathbb{R}^3 \to \mathbb{R}$ to be Lipschitz continuous functions with Lipschitz constant one.
Then,
\begin{align*}
\mathcal{E}_0 &:= \left| \frac{1}{n} \sum_{k=1}^n \psi^*(N_k, Q_k)  - \frac{1}{n} \sum_{k=1}^n \psi^*(\mathscr{N}_k, Q_k) \right| \\
&\leq \frac{1}{n} \sum_{k=1}^n \left| \psi^*(\mathscr{N}_k + 1, Q_k) - \psi^*(\mathscr{N}_k, Q_k) \right| 1(N_k = \mathscr{N}_k + 1) \\
&\leq \frac{1}{n} \sum_{k=1}^n 1(N_k = \mathscr{N}_k + 1) \leq \frac{|\Delta_n|}{n},
\end{align*}
and
\begin{align*}
\mathcal{E}_1 &:= \left| \sum_{k=1}^n \psi(N_k, Q_k, C_k) \frac{D_k}{L_n} - \sum_{k=1}^n \psi(\mathscr{N}_k, Q_k, \mathscr{C}_k) \frac{\mathscr{D}_k}{S_n} \right| \\
&\leq  \sum_{k=1}^n  \frac{\mathscr{D}_k}{S_n} \left| \psi(N_k, Q_k, \mathscr{C}_k) - \psi(\mathscr{N}_k, Q_k, \mathscr{C}_k) \right| 1(\Delta_n \leq 0) \\
&\hspace{5mm} +\sum_{k=1}^n \frac{D_k}{L_n}  \left|  \psi(\mathscr{N}_k, Q_k, C_k)  - \psi(\mathscr{N}_k, Q_k, \mathscr{C}_k)  \right| 1(\Delta_n > 0) \\
&\hspace{5mm} +  \sum_{k=1}^n \left|\psi(\mathscr{N}_k, Q_k, \zeta_k/\mathscr{D}_k) \left( \frac{D_k}{L_n} - \frac{\mathscr{D}_k}{S_n} \right)\right| 1(\Delta_n > 0) \\
&\leq  \sum_{k=1}^n  \frac{\mathscr{D}_k}{S_n} 1(N_k = \mathscr{N}_k +1)  + \sum_{k=1}^n \frac{D_k}{L_n} \left| \zeta_k/(\mathscr{D}_k+1) -\mathscr{C}_k \right| 1(D_k = \mathscr{D}_k + 1) \\
&\hspace{5mm} + \sum_{k=1}^n \left| \psi(\mathscr{N}_k, Q_k, \mathscr{C}_k) \right| \left| \frac{(D_k - \mathscr{D}_k) S_n - \mathscr{D}_k \Delta_n }{L_n S_n} \right| 1(\Delta_n > 0),
\end{align*}
where we used the fact that $\psi^*$ and $\psi$ have Lipschitz constant one. To bound further $\mathcal{E}_1$ use the Cauchy-Schwarz inequality to obtain
$$\sum_{k=1}^n  \frac{\mathscr{D}_k}{S_n} 1(N_k = \mathscr{N}_k +1) \leq \frac{n}{S_n} \left( \frac{1}{n} \sum_{k=1}^n \mathscr{D}_k^2 \right)^{1/2} \left( \frac{|\Delta_n|}{n} \right)^{1/2}.$$
Now, use the observation that $|\zeta_k| \leq c$ to obtain
\begin{align*}
& \sum_{k=1}^n \frac{D_k}{L_n} \left| \zeta_k/(\mathscr{D}_k+1) - \mathscr{C}_k \right| 1(D_k = \mathscr{D}_k + 1) \\
&\leq c \sum_{k=1}^n \frac{1}{L_n \mathscr{D}_k}  1(D_k = \mathscr{D}_k + 1, \mathscr{D}_k \geq 1)  + \sum_{k=1}^n \frac{1}{L_n} \left| \zeta_k - c \sgn(\zeta_k) \right| 1(D_k = \mathscr{D}_k + 1, \mathscr{D}_k = 0) \\
&\leq \frac{c}{L_n} \sum_{k=1}^n 1(D_k = \mathscr{D}_k+1) \leq \frac{c |\Delta_n|}{S_n}.
\end{align*}
Next, use the bound $|\psi(m,q,x)| \leq || (m,q,x) ||_1 + |\psi(0,0,0)|$ and H\"older's inequality to obtain
\begin{align*}
&\sum_{k=1}^n \left| \psi(\mathscr{N}_k, Q_k, \mathscr{C}_k) \right| \left| \frac{(D_k - \mathscr{D}_k) S_n - \mathscr{D}_k \Delta_n }{L_n S_n} \right| 1(\Delta_n > 0) \\
&\leq \sum_{k=1}^n \left| \psi(\mathscr{N}_k, Q_k, \mathscr{C}_k) \right| \frac{1(D_k = \mathscr{D}_k +1) }{S_n}  + \sum_{k=1}^n \left| \psi(\mathscr{N}_k, Q_k, \mathscr{C}_k) \right| \frac{ \mathscr{D}_k |\Delta_n| }{ S_n^2} \\
&\leq \frac{1}{S_n} \sum_{k=1}^n \left|\left| (\mathscr{N}_k, Q_k, c) \right|\right|_1 1(D_k = \mathscr{D}_k +1)  + \frac{|\Delta_n|}{S_n^2} \sum_{k=1}^n \left( \mathscr{N}_k \mathscr{D}_k + |Q_k| \mathscr{D}_k + c \right)   + \frac{2|\psi(0,0,0) \Delta_n| }{S_n} \\
&\leq \frac{n}{S_n} \left\{ \left( \frac{1}{n} \sum_{k=1}^n \mathscr{N}_k^{1+\delta} \right)^{1/(1+\delta)} +  \left( \frac{1}{n} \sum_{k=1}^n |Q_k|^{1+\delta} \right)^{1/(1+\delta)} \right\} \left( \frac{|\Delta_n|}{n} \right)^{\delta/(1+\delta)} \\
&\hspace{5mm} + \frac{|\Delta_n|}{S_n^2} \sum_{k=1}^n \left( \mathscr{N}_k \mathscr{D}_k + |Q_k| \mathscr{D}_k \right) + \frac{H |\Delta_n|}{S_n},
\end{align*}
where $0 < \delta < \min\{\alpha-1, \epsilon_Q\}$ and $H = 2|\psi(0,0,0)| + 2c$. Now note that since the bi-degree sequence is constructed on the event $|\Delta_n|\le n^{1-\kappa_0+\delta_0}$, we have that $\mathcal{E}_0 \leq n^{-\kappa_0 + \delta_0}$ a.s. To show that $\mathcal{E}_1$ converges to zero a.s. use the Strong Law of Large Numbers (SLLN) (recall that $E[\mathscr{D}^2] < \infty$ and that $\mathscr{N}, \mathscr{D}, Q$ are mutually independent) and the bounds derived above.

Finally, by the SLLN again and the fact that $E[ ||(\mathscr{N}, Q, \mathscr{C}) ||_1] < \infty$, we have
$$\lim_{n \to \infty} \frac{1}{n}\sum_{k=1}^n \psi^*(N_k ,Q_k) = \lim_{n \to \infty} \frac{1}{n}\sum_{k=1}^n \psi^*(\mathscr N_k ,Q_k) = E[ \psi^*( \mathscr{N}, Q)] \quad \text{a.s.}$$
and
$$\lim_{n \to \infty} \sum_{i=1}^n \psi(N_k, Q_k, C_k) \frac{\mathscr D_i}{S_n} = \lim_{n \to \infty} \sum_{k=1}^n \psi(\mathscr{N}_k, Q_k, \mathscr{C}_k) \frac{\mathscr{D}_k}{S_n} =  \frac{1}{\mu} E[ \psi(\mathscr{N}, Q, \mathscr{C}) \mathscr{D}] \quad \text{a.s.}$$
The first limit combined with the duality formula gives that $d_1(F_n^*, F^*) \to 0$ a.s. For the second limit we still need to identify the limiting distribution, for which we note that
\begin{align*}
\frac{1}{\mu} E[\psi(\mathscr{N}, Q, \mathscr{C}) \mathscr{D}] &= \frac{1}{\mu} E\left[ E[ \psi(\mathscr{N}, Q, \mathscr{C}) \mathscr{D} | \mathscr{N}, Q] \right] = \frac{1}{\mu} E\left[ \sum_{i=1}^\infty \int_{-\infty}^\infty \psi(\mathscr{N}, Q, z/i) i \, d F^\zeta(z) P(\mathscr{D} = i)\right] \\
&= \frac{1}{\mu} E\left[ \sum_{i=1}^\infty \int_{-\infty}^{\infty} \psi(\mathscr{N}, Q, y) i \, d F^\zeta(yi) P(\mathscr{D} = i)\right] =: E\left[ \psi(\mathscr{N}, Q, Y) \right],
\end{align*}
where $Y$ has distribution function
\begin{align*}
P(Y \leq x) &= \frac{1}{\mu} E\left[ \sum_{i=1}^\infty \int_{-\infty}^\infty 1(y \leq x) i \, d F^\zeta (yi) P(\mathscr{D} = i)\right] = \frac{1}{\mu} E\left[ \sum_{i=1}^\infty i F^\zeta(ix) P(\mathscr{D} = i)\right] \\
&= \frac{1}{\mu} E[ \mathscr{D} F^\zeta( \mathscr{D}x) ] = \frac{1}{\mu} E[ \mathscr{D} 1(\zeta/\mathscr{D} \leq x)] = P(\mathcal{C} \leq x).
\end{align*}
It follows that $E[ \psi(\mathscr{N}, Q, \mathscr{C}) \mathscr{D}] /\mu = E[\psi(\mathscr{N}, Q, \mathcal{C}) ]$, which combined with the duality formula gives that $d_1(F_n, F) \to 0$ a.s.
\end{proof}

\bibliographystyle{plain}
\bibliography{PageRankBib}

\end{document}